\documentclass[a4paper,11pt,reqno]{amsart}
\usepackage{amsmath}
\usepackage{relsize}
\usepackage[noadjust]{cite}
\usepackage{color}
\usepackage{graphicx}
\usepackage[dvipsnames]{xcolor}

\usepackage{hyperref, enumitem}
\hypersetup{
  colorlinks   = true, 
  urlcolor     = blue, 
  linkcolor    = Purple, 
  citecolor   = black 
}
\usepackage{amsmath,amsthm,amssymb}
\usepackage{setspace}
\usepackage{enumitem}
\setitemize{itemsep=-1pt}
\setenumerate{itemsep=-1pt}
\usepackage[margin=2.5cm]{geometry}

\usepackage{stmaryrd}
\usepackage{longtable}

\usepackage{array}
\newcolumntype{x}[1]{>{\centering\arraybackslash\hspace{0pt}}p{#1}}

\theoremstyle{definition}
\newtheorem{theorem}{Theorem}[section]
\newtheorem{definition}[theorem]{{{Definition}}}
\newtheorem{example}[theorem]{{{Example}}}
\newtheorem{notation}[theorem]{{{Notation}}}
\newtheorem{remark}[theorem]{{{Remark}}}

\newtheorem{corollary}[theorem]{{{Corollary}}}
\newtheorem{proposition}[theorem]{{{Proposition}}}
\newtheorem{lemma}[theorem]{{{Lemma}}}

\newcommand{\numberset}{\mathbb}
\newcommand{\N}{\numberset{N}}

\newcommand{\Z}{\numberset{Z}}
\newcommand{\ZZ}{\numberset{Z}}
\newcommand{\Q}{\numberset{Q}}
\newcommand{\R}{\numberset{R}}

\newcommand{\F}{\numberset{F}}

\newcommand{\T}{\mathcal{T}}
\newcommand{\mA}{\mathcal{A}}
\newcommand{\C}{\mathcal{C}}
\newcommand{\mC}{\mathcal{C}}
\newcommand{\G}{\mathcal{G}}
\newcommand{\mP}{\mathcal{P}}
\newcommand{\B}{\mathcal{B}}
\newcommand{\M}{\mathcal{M}}

\newcommand{\mU}{\mathcal{U}}
\newcommand{\U}{\mathcal{U}}

\newcommand{\kL}{\mathcal{L}}
\newcommand{\X}{\mathcal{X}}

\newcommand{\mW}{\mathcal{W}}

\newcommand{\mH}{\mathcal{H}}

\newcommand{\Fq}{\F_q}

\newcommand{\colsp}{\textnormal{colsp}}
\newcommand{\col}{\textnormal{colsp}}

\newcommand{\mL}{\mathcal{L}}

\newcommand{\Fqm}{\mathbb{F}_{q^m}}

\newcommand{\pp}{\mathbb{P}}

\newcommand\qbin[3]{\left[\begin{matrix} #1 \\ #2 \end{matrix} \right]_{#3}}

\newcommand{\mS}{\mathcal{S}}

\newcommand{\h}{\textnormal{h}}
\newcommand{\Tr}{\textnormal{Tr}}

\newcommand{\rk}{\textnormal{rk}}

\newcommand{\crit}{\textnormal{crit}}

\newcommand{\fqm}{\mathbb{F}_{q^m}}

\DeclareMathOperator{\zero}{\mathbf{0}}
\DeclareMathOperator{\one}{\mathbf{1}}
\DeclareMathOperator{\GL}{GL}
\DeclareMathOperator{\supp}{supp}

\DeclareMathOperator{\PG}{PG}
\DeclareMathOperator{\cl}{cl}

\DeclareMathOperator{\lker}{lker}

\DeclareMathOperator{\rowsp}{rowsp}

\title{Recursive properties of the characteristic polynomial of weighted lattices}
 \usepackage[foot]{amsaddr}
 \author{Gianira N. Alfarano$^1$}
 \author{Eimear Byrne$^2$}
 \address{$^1$Université de Rennes, France.}
 \address{$^2$University College Dublin, Ireland.}
 \email{gianira-nicoletta.alfarano@univ-rennes.fr, ebyrne@ucd.ie}

\begin{document}

\begin{abstract}
 In this paper, we describe properties of the characteristic polynomial of a weighted lattice and show that it has a recursive description, which we use to obtain results on the critical exponent of $q$-polymatroids. We give a Critical Theorem for representable $q$-polymatroids and we provide a lower bound on the critical exponent. We show that $q$-polymatroids arising from certain families of rank-metric codes attain this lower bound.
\end{abstract}

\maketitle

\section{Introduction}
The characteristic polynomial is a well-studied invariant of matroid theory and its generalisations, which has been used to establish duality theorems for codes, matroids and $q$-polymatroids; see \cite{britzshiromotomacwill,byrneweighted,greene}. In the case of matroids, the characteristic polynomial is a generalised Tutte-Grothendieck (T-G) invariant and can be obtained by taking evaluations of the Tutte and rank generating polynomials; see~\cite{bryl_oxley}. The characteristic polynomial of a {\em weighted lattice} and its connection to critical problems was studied by Whittle in \cite{whittle,whittlecriticalpoly}, thus unifying the theory for matroids, polymatroids, geometric lattices and graded lattices. 

Critical problems concerning these structures are highly related to the characteristic polynomial.
The critical exponent of a collection $\mS$ of $n$ points in the projective space $\PG(k-1,q)$ is defined to be the minimum number of hyperplanes whose intersection with $\mS$ is empty. A reformulation of this definition from the point of view of coding theory can be read in \cite{greene}: the critical exponent of the matroid arising from a code $C\leq \F_q^n$ is the minimum dimension of a subcode of $C$ whose support is $\{1,\ldots,n\}$. 
The critical problem for a representable matroid is the determination of its critical exponent and is strongly related to the celebrated MDS conjecture; see \cite{segre1955curve}. 
In \cite{CrapoRota}, Crapo and Rota proved that the critical exponent of a collection of vectors or projective points is actually equal to an evaluation of the characteristic polynomial of the matroid represented by those points, which is a result known as the Critical Theorem. This theorem provides a more general setting for many results in extremal combinatorics
 and has been the subject of numerous applications and generalisations; see \cite{alon,britz2005extensions,bryl_oxley,tutte_1954,zaslavsky1987mobius,koga+,IMAMURA,gruica2022rank, imamura2023}.

 In this paper, we consider an extension of these topics. 
 We consider the characteristic polynomial of a weighted lattice and study its recursive properties in relation to the characteristic polynomials of its minors. Unlike the approach of Whittle in \cite{whittle}, the minors we consider retain the weighted lattice properties of the original weighted lattice. In exchanging the role of the Boolean lattice for an arbitrary lower semimodular lattice~$\mL$ in the theory of polymatroids we arrive at $\mL$-polymatroids. In the special case that $\mL$ is the lattice of subspaces of a vector space over $\F_q$, we have a $q$-polymatroid. Such objects have many connections to rank-metric codes. Such expressions can be used to obtain results on the critical exponent of a representable $q$-polymatroid, which is one that arises from a matrix code. 
 In specialising these results to the Boolean lattice, one recovers familiar identities for matroids and polymatroids; see e.g.~\cite{asano,bryl_oxley}. 
  We provide a lower bound on the critical exponent of a representable $q$-polymatroid and identify some classes of matrix codes that meet this lower bound. 
 
 This paper is organised as follows. In section \ref{sec:preliminaries}, we provide the necessary background material on lattices, weighted lattices, $q$-polymatroids and rank-metric codes. In section~\ref{sec:charpoly}, we derive expressions of the characteristic polynomial ${\mathbb P}(\mW;z)$ of a weighted lattice $\mW$ in terms of characteristic polynomials of the minors of $\mW$. We furthermore establish criteria for the positivity of evaluations of the characteristic polynomials of some of the minors of $\mW$ in terms of evaluations of ${\mathbb P}(\mW;z)$.
 In section~\ref{sec:crit_theorem}, we give a $q$-analogue of the Critical Theorem and hence define the critical exponent of a representable $q$-polymatroid. We describe a geometric interpretation of the critical exponent for vector rank-metric codes. We apply the results of section~\ref{sec:charpoly} to obtain $q$-analogues of results on critical exponents such as those found in \cite{asano}.
 In section~\ref{sec:lower_bound}, we give a lower bound on the critical exponent of a representable $q$-polymatroid and we give sufficient conditions that ensure this bound is met with equality for certain families of codes. In section~\ref{sec:gen_britz}, we consider the $q$-analogue of several generalisations of the classical Critical Theorem, such as those given in \cite{britz2005extensions}.

\medskip

\paragraph{\textbf{Notation}}
    Throughout this paper, $q$ is a prime power, while $n$ and $m$ are integers satisfying $n,m\geq 2$. We denote by $\F_q$ the finite field with $q$ elements and by $\F_{q^m}$ its field extension of degree $m$.
    We let $E$ denote an $n$-dimensional vector space over~$\F_q$ and we let $\{e_1,\ldots,e_n\}$ denote the standard basis of $E$. We let $\F_{q}^{n \times m}$ denote the space of $n\times m$ matrices with entries in $\F_q$. For a matrix $M\in\F_q^{n\times m}$ we denote by $\colsp(M)$ the column-space of $M$ over~$\F_q$.
    For any positive integer $s$, we define $[s]:=\{1,\dots,s\}$.

\section{Preliminaries}\label{sec:preliminaries}

In this section, we introduce $q$-polymatroids, rank-metric codes and we collect some well-known facts about lattices.

\subsection{Lattices}
In this short subsection we recall some preliminary results on lattices. 
We refer the reader to~\cite{birkhoff1940lattice,romanlattice,stanley2011enumerative} for further reading on ordered lattices.

\begin{definition}
	Let $(\mathcal{L}, \leq)$ be a partially ordered set (poset). Let $a,b,v \in \mathcal{L}$.
	We say that $v$ is an \textbf{upper bound} of $a$ and $b$ if $a\leq v$ and $b\leq v$ and furthermore, we say	that $v$ is a \textbf{least upper bound} of $a$ and $b$ if $v \leq u$ for any $u \in  \mathcal{L} $ that is also an upper bound of $a$ and $b$.  
	If a least upper bound of $a$ and $b$ exists, then it is unique, is denoted by $a\vee b$, which is called the \textbf{join} of $a$ and $b$.
	We analogously define \textbf{a lower bound} and \textbf{the greatest lower bound} of $a$ and $b$ and denote the unique greatest lower bound of $a$ and $b$ by $a\wedge b$, which is called the \textbf{meet} of $a$ and $b$.
    The poset $\mathcal{L}$ is called a \textbf{lattice} if each pair of elements has a least upper bound and a greatest lower bound and it is denoted by $(\mathcal{L}, \leq, \vee, \wedge)$. 
    An element in $\mL$ that is not smaller than any other element is called \textbf{maximal} element of $\mL$ and it is denoted by $\boldsymbol{1}_{\mL}$ and an element that is not bigger than any other element is called \textbf{minimal} element of $\mL$ and it is denoted by~$\boldsymbol{0}_{\mL}$. If there is no risk of confusion, we simply write $\boldsymbol{1}$ and $\boldsymbol{0}$.
\end{definition}

\begin{definition}
    For $i \in \{1,2\}$, let $(\mathcal{L}_i, \leq_i, \vee_i, \wedge_i)$ be a lattice. 
    A bijection $\phi:\mL_1 \longrightarrow \mL_2$ is~called 
    \begin{enumerate}
        \item a \textbf{lattice isomorphism} if
    $\phi(a \vee_1 b) = \phi(a) \vee_2 \phi(b)$ and $\phi(a\wedge_1b) = \phi(a) \wedge_2 \phi(b)$ for all $a,b \in \mL_1$;
        \item 
        a \textbf{lattice anti-isomorphism} if
    $\phi(a \vee_1 b) = \phi(a) \wedge_2 \phi(b)$ and $\phi(a\wedge_1b) = \phi(a) \vee_2 \phi(b)$ for all $a,b \in \mL_1$.
    \end{enumerate}
    
\end{definition}

\begin{definition}
Let $\mathcal{L}$ be a lattice with meet $\wedge$ and join $\vee$.
Let $a, b\in\mathcal{L}$. 
\begin{enumerate}
    \item Let $a\leq b$. The {\bf interval} $[a,b]\subseteq\mathcal{L}$ is the set of all $x\in\mathcal{L}$ such that $a\leq x\leq b$. It defines the {\bf interval sublattice} $([a,b],\leq,\vee,\wedge)$.
    \item Let $a \leq b$ and let $c\in  [a,b]$. We say that $d$ is a \textbf{complement} of $c$ in $[a,b]$  if $c \wedge d = a$ and $c \vee d = b$. 
    \item $\mL$ is called is called {\bf complemented} if every $c \in \mL$ has a complement in $\mL$. $\mL$ is called {\bf relatively complemented} if every interval of $\mL$ is complemented.
    \item Let $a\leq b$. If $[a,b]\subseteq \mL$ is such that for any $x\in \mL$, $x\in[a,b]$ implies that $x=a$ or $x=b$, then $b$ is called a {\bf cover} of $a$ and we write $a\lessdot b$. 
    \item An {\bf atom} or {\bf point} of $\mL$ is any element that is a cover of $\textbf{0}$. A {\bf coatom} or {\bf copoint} of $\mL$ is any element that is covered by \textbf{1}. We define $\mA([a,b]):=\{ x \in [a,b]: a \lessdot x \}$ and 
    $\mH([a,b]):=\{x \in [a,b] : x \lessdot b\}$. We also define
    $\mA(b):=\{x \in [\zero,b] : \zero \lessdot \;x \}$ and $\mH(b):=\{x \in [\zero,b] : x \lessdot b\}$.
    \item 
    A {\bf chain} from $a$ to $b$ is a totally ordered subset of $[a,b]$ with respect to $\leq$. A chain from $a$ to $b$ is called a {\bf maximal} chain in $[a,b]$ if it is not properly contained in any other chain from $a$ to $b$. 
    A finite {\bf chain} from $a$ to $b$ is a sequence of the form
    $ a = x_1 < \cdots < x_{k+1}=b $ with $x_j\in\mathcal{L}$ for $j \in [k]$, in which case we say that the chain has length $k$. 
    \item Let $a \leq b$. If $\mL$ has no infinite chains, we define the {\bf length} of $[a,b]$, denoted $\ell([a,b])$, to be the maximum length over all maximal chains from $a$ to $b$, if this maximum is finite. Otherwise, $\ell([a,b])$ is infinite. The {\bf height} of $b$ is defined to be $\h(b):=\ell([\zero,b])$.
    \item $\mL$ is called \textbf{lower} (resp. \textbf{upper}) {\bf semimodular} if $a \lessdot a \vee b$ implies $a \wedge b \lessdot b $ (resp. $a \wedge b \lessdot b$ implies $a \lessdot a \vee b$).
    \item $\mL$ is called {\bf modular} if for all $a,b,c \in \mL$, we have that
$a \geq c$ implies $(a \wedge b) \vee c = a \wedge (b \vee c).$
\end{enumerate}
\end{definition}

Every modular lattice is both upper and lower semimodular, while the converse holds for lattices of finite length.
If $\mL$ has no infinite chains and is either lower or upper semimodular, then it has the Jordan-Dedekind chain property, which means that any pair maximal chains from $A$ to $B$ in $\mL$ have the same length. Furthermore, this implies that $\mL$ is graded with respect to its height function, i.e., if $X,Y \in \mL$ and $X \lessdot Y$, then $\h(Y)=\h(X)+1$. Finally, if $\mL$ is lower (respectively, upper) semimodular, then for any anti-isomorphism $\varphi$ on $\mL$, we have that $\varphi(\mL)$ is upper (respectively, lower) semimodular. The interested reader is referred to \cite{romanlattice} for further details.

\begin{notation}
  Throughout this paper, unless stated otherwise, $\mL$ will denote a 
  lattice of finite length $n$, with order relation $\leq$, meet $\wedge$, and join $\vee$. We will denote atoms of $\mL$ in lowercase, while if the height of an element is not specified to be 1, we will use uppercase. 
  We write $\mH$ to denote the coatoms of $\mL$ and we write $\mA$ to denote the set of atoms of $\mL$.
\end{notation}

With respect to the connections with coding theory discussed in this paper, we are particularly interested in the subspace lattice $(\mathcal{L}(E), \leq, \vee, \wedge)$, which is the lattice of $\mathbb{F}_q$-subspaces of $E$, ordered with respect to inclusion and for which the join of a pair of subspaces is their vector space sum and the meet of a pair of subspaces is their intersection. That is, for all subspaces $A,B \leq E$ we have:
$ A \vee B = A + B, \; A \wedge B = A \cap B.$
The minimal element of $\mL(E)$ is $\zero=\langle 0 \rangle$ and its maximal element is $\textbf{1}=E$. 
$\mL(E)$ is a modular, complemented lattice. 
For each $U \in \mL(E)$, we write $U^\perp$ to denote the orthogonal complement of $U$ with respect to a fixed non-degenerate bilinear form on $E$. The map $U\mapsto U^\perp$ is an involutory anti-automorphism of~$\mL(E)$. We let $\mA(E)$ and $\mH(E)$ denote the atoms and coatoms, respectively, of $\mL(E)$.

We may think of $\mL(E)$ as a $q$-analogue of the Boolean lattice $\mL(S)$, where $S$ is a set of cardinality $n$. 
In $\mL(S)$, the meet operation is again intersection while the join operation is set-theoretic union. Several formulae and invariants related to the objects associated with $\mL(E)$ can be obtained in the Boolean case by setting $q=1$.  

We recall the definition of a closure operator on a lattice; see \cite{crapo1969mobius,whittle}.

\begin{definition}
   A {\bf closure operator} on the lattice $\mL$ is a function $\text{cl}: \mL \longrightarrow \mL$ such that
   \begin{enumerate}
       \item $\cl(X) \geq X$, for all $X \in \mL$; 
       \item $\cl(X) \leq \cl(Y)$ for all $X,Y \in \mL$ such that $X \leq Y$;
       \item $\cl(\cl(X))=\cl(X)$ for all $X \in \mL$. 
   \end{enumerate}
\end{definition}

Let $\cl$ be a closure operator on $\mL$. An element $X \in \mL$ is called {\bf closed} with respect to $\cl$ if $X=\cl(X)$. The quotient lattice of $\mL$ with respect to $\cl$ is the lattice of closed elements of $\mL$. 

\begin{definition}[\cite{whittle}]
    Let $f:\mL \longrightarrow \N_0$ such that $f(\zero)=0$ and for all $A,B\in \mL$, $f(A) \leq f(B)$ whenever $A \leq B$. We say that the pair $(\mL,f)$ is an (integer) {\bf weighted lattice}.
\end{definition}

\begin{definition}
    Let $\mW=(\mL,f)$ be a weighted lattice and let $X,Y \in \mL$.
    Define the map $f_{[X,Y]}: \mL \longrightarrow \N_0$ by 
    $f_{[X,Y]}(T) = f(T)-f(X)$ for every $T \in [X,Y]$.
Then $\mW([X,Y])$ is the weighted lattice defined by the pair $([X,Y],f_{[X,Y]})$.
The weighted lattice $\mW([X,Y])$ is called a {\bf minor} of $\mW$.
\end{definition}

\begin{definition}
 We say that two weighted lattices $\mW_1=(\mL_1,f_1)$ and 
$\mW_2=(\mL_2,f_2)$ are {\bf scaling-lattice-equivalent}
if there exists a lattice isomorphism $\phi: \mL_1\longrightarrow \mL_2$ and $\lambda \in \Q$ such that $f_1(A) = \lambda f_2(\phi(A))$ for all $A \in \mL_1$. If $\lambda =1$, we say that $\mW_1$ and $\mW_2$ are {\bf lattice-equivalent} and we write $\mW_1 \cong \mW_2$. 
\end{definition}

We mention some other well-known facts on weighted lattices; see \cite{whittle} for a more detailed treatment.
  The element $F \in \mL$ is called a {\bf flat} of the weighted lattice $(\mL,f)$ if $f(A)>f(F)$ whenever
    $A>F$.
    A closure operator $\cl$ on $\mL$ is called {\bf respectful} on $(\mL,f)$ if $f(A)=f(\cl(A))$ for all $A \in \mL$. A closure operator is respectful on $(\mL,f)$ if and only if every flat $F$ of $(\mL,f)$ satisfies $\cl(F)=F$, i.e. is closed. 
    The collection of flats of $(\mL,f)$ forms a lattice called the \textbf{lattice of flats} of $(\mL,f)$. We define the {\bf principal closure} of $A \in \mL$ to be the meet of all the flats of $(\mL,f)$ that contain $A$. Thus the quotient lattice with respect to the principal closure operator on $(\mL,f)$ coincides with the lattice of flats of $(\mL,f)$. 
We say that an atom $x \in \mL$ is a {\bf loop} of the weighted lattice $(\mL,f)$ if $f(x)=0$.

The M\"obius function (see, e.g. ~\cite[Chapter 25]{Lint}) is necessary for the definition of a characteristic polynomial. 

\begin{definition}
Let $(\mP,\leq)$ be a partially ordered set. The M\"obius function for $\mP$ is defined via the recursive formula for all $X,Y \in \mP$: 
\begin{align*}
\mu(X,Y) := \left\{ 
        \begin{array}{cl}
        1 & \text{ if } X=Y,\\
        -\sum\limits_{{X\leq Z<Y}}\mu(X,Z) & \text{ if } X<Y, \\
        0 & \text{otherwise}.
        \end{array}
          \right.
\end{align*}
\end{definition}
An equivalent definition of $\mu$ is given by:
\begin{equation}\label{eq:mob}
     \sum_{X \leq  Z \leq Y} \mu(Z,Y) =0 \:\:\forall X,Y \in \mP, \; X < Y \text{ and }\mu(X,X)=1 \:\:\forall X \in \mL.
\end{equation}

\begin{lemma}[The M\"obius Inversion Formula]
 Let $f,g:\mP\longrightarrow \mathbb{Z}$ be functions on a poset $\mP$. The following hold.
	\begin{enumerate}
		\item$\displaystyle f(X)=\sum_{X\leq Y}g(Y)\:\forall\:X \in \mP \mbox{ if and only if }g(X)=\sum_{X\leq Y}\mu(X,Y)f(Y) \:\forall\:X \in \mP$.
		\item$\displaystyle f(X)=\sum_{X\geq Y}g(Y)\:\forall\:X \in \mP\mbox{ if and only if }g(X)=\sum_{X\geq Y}\mu(Y,X)f(Y)\:\forall\:X \in \mP$.
	\end{enumerate}
\end{lemma}

We recall the following result of Crapo \cite[Theorem 1]{crapo1969mobius}.

\begin{lemma}[\cite{crapo1969mobius}]\label{lem:crapo}
   Let $\mL'$ be the quotient lattice of $\mL$ with respect to a closure operator $\cl$ of~$\mL$ and let $\mu'$ be its corresponding M\"obius function. Let $X,Y \in \mL$. Then 
    \[
       \sum_{Z \in \mL: \cl(Z) = \cl(Y)}\mu(X,Z) 
       = \left\{ \begin{array}{cc}
            \mu'(\cl(X),\cl(Y))& \text{ if } \cl(X) = X,   \\
            0 & \text{ if } \cl(X) > X.
         \end{array}\right.
    \]
\end{lemma}

Given $U,V \in \mL(E)$ of dimensions $\text{u}$ and $\text{v}$, respectively, we have that
\begin{equation*}\label{eq:mobinv}
\mu\left(U,V\right)=\left\lbrace
\begin{array}{cl}
(-1)^{\text{v}-\text{u}}q^{\binom{\text{v}-\text{u}}{2}}&\mbox{ if }U\leq V,\\
\\
0 & \mbox{ otherwise.}
\end{array}
\right.
\end{equation*}


\subsection{Basic Notions of $\mL$-Polymatroids and {\em q}-Polymatroids}

We outline basic facts and definitions of $\mL$-Polymatroids and $q$-polymatroids. In all of the following, the definitions are extensions of those found in the literature on matroids and polymatroids; we simply consider these structures to be defined on more general lattices than the Boolean lattice.

For structures defined on the subspace lattice, the main difference to those defined over the Boolean lattice is brought about by the fact that the latter is distributive, while the subspace lattice is not. Furthermore, while both lattices are relatively complemented, the complement of an element of the subspace lattice is not unique. Finally, over a finite field, an orthogonal complement of an element is not in general one of its complements in the subspace lattice. 
The first results on $q$-polymatroids and codes can be read in \cite{shiromoto19} and \cite{gorla2019rank}.

\begin{definition}
An \textbf{$(\mL,r)$-(integer) polymatroid} is a pair $\M=(\mL,\rho)$ for which $r \in \N_0$ and $\rho$ is a function $\rho: \kL \longrightarrow {\mathbb N}_0$ satisfying the following axioms.

\begin{itemize}
	\item[(R1)] Boundedness: $0\leq \rho(A) \leq r \cdot \h(A)$, for all $A \in \mL$.
	\item[(R2)] Monotonicity: $A\leq B \Rightarrow \rho(A)\leq \rho(B)$,  for all $A,B \in \mL$.
	\item[(R3)] Submodularity: $\rho(A \vee B)+\rho(A\wedge B)\leq \rho(A) +\rho(B)$, for all $A,B \in \mL$.
\end{itemize}
If $\mL=\mL(E)$, we say that $\M$ is a $(q,r)$-polymatroid.
\end{definition}
If it is not necessary to specify $r$, we will simply refer to such an object as an $\mL$-polymatroid, in the case of arbitrary $\mL$, or as a $q$-polymatroid for $\mL=\mL(E)$. We define a \textbf{$q$-matroid} to be a $(q,1)$-polymatroid. A $(1,r)$-polymatroid is an integer polymatroid and a $(1,1)$-polymatroid is simply a matroid.
Clearly, an $\mL$-polymatroid $(\mL, \rho)$ is a weighted lattice whose underlying function $\rho$ is bounded and submodular.
We say that an atom $x \in \mL$ is a {\bf loop} of the $\mL$-polymatroid $\M=(\mL,\rho)$ if $\rho(x)=0$. A coatom $H \in \mL$ is called a {\bf coloop} of $\M$ if $\rho(H) = \rho(\one)-r$.

\begin{lemma}\label{lem:maxr}
    Let $\mL$ be a lower semimodular or an upper semimodular lattice. Let $\M=(\mL,\rho)$ be an $\mL$-polymatroid. Let $[X,Y]$ be an interval of $\mL$ and let $r=\max\{ \rho(T)-\rho(S): S,T \in [X,Y], S \lessdot T\}$.
    Then $r(\h(Y)-\h(X)) \geq \rho(Y)-\rho(X)$.
\end{lemma}
\begin{proof}
   Let $X =X_0 \lessdot \cdots \lessdot X_k =Y$ be a maximal chain from $X$ to $Y$.
   Since $\mL$ satisfies the Jordan-Dedekind chain condition, we have that 
   $\h(Y)-\h(X)=k$. By our choice of $r$, we thus have that:
   \[
   \rho(Y) - \rho(X) = \rho(X_k) - \rho(X_{k-1})+ \rho(X_{k-1})-\cdots + \rho(X_1) -\rho(X_0) 
   \leq r k = r(\h(Y) - \h(X)).
   \]
\end{proof}

\begin{proposition}\label{prop:minors}
    Let $\mL$ be a lower semimodular or an upper semimodular lattice. Let $\M=(\mL,\rho)$ be an $\mL$-polymatroid. Let $[X,Y]$ be an interval of $\mL$ and let $r=\max\{ \rho(T)-\rho(S): S,T \in [X,Y], S \lessdot T\}$. 
    Then $([X,Y],\rho_{[X,Y]})$ is an $(\mL,r)$-polymatroid.
\end{proposition}
\begin{proof}
   It is clear that (R2) and (R3) hold. Furthermore, it is easy to see that $\rho_{[X,Y]}(U) = \rho(U) - \rho(X) \geq 0$ for all $U \in [X,Y]$.
   By Lemma \ref{lem:maxr}, for any $U\in [X,Y]$ we have that:
   \[\rho_{[X,Y]}(U) \leq r k = r(\h(U) - \h(X)),\]
   from which it follows that (R1) holds for the given choice of $r$.
\end{proof}
\begin{remark}
    Note that the assumption that $\mL$ is lower or upper semimodular is required only to show that $\rho_{[X,Y]}(U) \leq r \cdot \h(U)$ for every $U \in [X,Y]$, where $r=\max\{ \rho(T)-\rho(S): S,T \in [X,Y], S \lessdot T\}$. If $\mL$ is an arbitrary finite lattice equipped with a rank function $\rho$ satisfying axioms (R1)-(R3), then for any interval $[X,Y]$ of $\mL$, $\rho_{[X,Y]}$ is the rank function of an $([X,Y],r')$-polymatroid for some integer $r'$. 
\end{remark}
We therefore define the following {\em minors} of an 
$\mL$-polymatroid.

\begin{definition}
Let $\M=(\mL,\rho)$ be an $\mL$-polymatroid and let $[X,Y]$ be an interval of~$\mL$. 
Then $\M([X,Y])$ is the $\mL$-polymatroid defined by the pair 
$([X,Y],\rho_{[X,Y]})$. We say that $\M([X,Y])$ is a {\bf minor} of $\M$.  
Let $\varphi$ be an anti-automorphism of $\mL$.
\begin{enumerate}
    \item 
    We write $\M|_Y:=\M([\zero,Y])$, which is called the {\bf restriction} of $\M$ to $Y$.
   \item 
   We write $\M/X:=\M([X,E])$, which is called the {\bf contraction} of $\M$ {\bf by} $X$. 
  \item 
  For $T \in \mL$, we write $\M.T:=\M/\varphi(T)$, which is called the {\bf contraction} of $\M$ {\bf to}~$T$.
\end{enumerate}
\end{definition}

In the more general context of a weighted lattice $\mW=(\mL,f)$, we similarly write $\mW|_Y:=\mW([\zero,Y])$ and $\mW/X:=\mW([X,\one])$ for all $X,Y \in \mL$.

\begin{proposition}
    Let $\mL$ be a lower semimodular lattice. Let $\M=(\mL,\rho)$ be an $\mL$-polymatroid. Let $\mL^*$ be a lattice of finite length $n$ such that there exists a lattice anti-isomorphism $\varphi: \mL \longrightarrow \mL^*$. Let $r=\max\{ \rho(T)-\rho(S): S,T \in \mL, S \lessdot T\}$. For every $A\in \mL$, define 
    $\rho^*(\varphi(A)):=r \cdot \h(\varphi(A))-\rho(\one)+\rho(A)$.
Then $\M^*=(\mL^*,\rho^*)$ is an $(\mL^*,r)$-polymatroid called the {{\bf dual}} of~$\M$.
\end{proposition}
\begin{proof}
    By Lemma \ref{lem:maxr}, we have that $r(\h(B)-\h(A))\geq \rho(B)-\rho(A)$ for any interval $[A,B]$ of $\mL$. 
    It follows that for any $A\in \mL$, we have $\rho^*(\varphi(A)):=r \cdot \h(\varphi(A))-\rho(\one)+\rho(A) = r(\h(\one)-\h(A))-(\rho(\one)-\rho(A))\geq 0$. Moreover, since $\rho$ is increasing on $\mL$ we have that $\rho(\one) \geq \rho(A)$ and so $\rho^*(\varphi(A)) \leq r\cdot \h(\varphi(A)) $. Therefore, (R1) holds.
    Let $[A,B]$ be an interval of $\mL$. Then $[\varphi(B),\varphi(A)]$ is an interval of $\mL^*$ and 
    \begin{eqnarray*}
        \rho(\varphi(A))-\rho(\varphi(B)) & = & 
        r(\h(\varphi(A)-\h(\varphi(B)) +\rho(B)-\rho(A)\\
        &=& r\ell([\varphi(B),\varphi(A)]) + \rho(B)-\rho(A)\\
        &=& r\ell([A,B]) + \rho(B)-\rho(A)\\
        &=& r(\h(B)-\h(A)) +\rho(B)-\rho(A),
    \end{eqnarray*}
    which is non-negative by Lemma \ref{lem:maxr} and shows that (R2) holds.
    Finally, we show that (R3) holds. 
    Since
    $\mL$ is lower semimodular, $\mL^*$ is upper semimodular and hence
    for any $S,T \in \mL^*$ we have
    $\h(S \vee T)+\h(S \wedge T) \leq \h(S) +\h(T)$. Therefore, for any $A,B \in \mL$ we have: 
    \begin{eqnarray*}
        & &\rho^*(\varphi(A) \vee \varphi(B)) + \rho^*(\varphi(A) \wedge \varphi(B)) \\
        &=& r \cdot \h(\varphi(A) \vee \varphi(B))-\rho(\one)+\rho(A\wedge B) + r \cdot \h(\varphi(A) \wedge \varphi(B))-\rho(\one)+\rho(A\vee B)\\
        &\leq & r\cdot \h(\varphi(A)) -\rho(\one) + \rho(A) +r\cdot \h(\varphi(B)) -\rho(\one) + \rho(B)\\
        &=& \rho^*(\varphi(A)) + \rho^*(\varphi(B)).
    \end{eqnarray*}
\end{proof}

Different choices of anti-isomorphism $\varphi$ may lead to different dual matroids, but all duals of an $\mL$-polymatroid $\M$ are lattice-equivalent.

\begin{definition}
    Let $\mW=(\mL,f)$ be a weighted lattice and 
    let $X,Y \in \mL$. Let $\mu$ be the M\"obius function of $\mL$.
	The {\bf characteristic polynomial} of $\mW([X,Y])$ is the polynomial in $\ZZ[z]$ defined~by 
   $$ \pp(\mW([X,Y]);z) := \sum_{A\in [X,Y]} \mu(X,A) z^{f(Y)-f(A)}.$$ 
   In particular, we have
	$$ \pp(\mW;z) := \sum_{X\in \mL} \mu(\textbf{0},X) z^{f(\one)-f(X)}.$$
\end{definition}
It is straightforward to check that the characteristic polynomial is an invariant of the lattice-equivalence class of a weighted lattice. 
Indeed, if $\phi: \mL \longrightarrow \mL'$ is a lattice isomorphism, such that $\mW=(\mL,f)$ and $\mW'=(\mL',f')$ are lattice-equivalent via $f(X) = f'(\phi(X))$ for all $X \in \mL$, then

\begin{align*}
    \pp(\mW;z) = \sum_{X\in \mL} \mu(\textbf{0},X) z^{f(\one)-f(X)}
    =\sum_{X\in \mL} \mu(\textbf{0},\phi(X)) z^{f'(\phi(\one))-f'(\phi(X))}
    = \pp(\mW';z).
\end{align*}

If $\mW=(\mL,f)$ and $\mW'=(\mL',f')$ are scaling-lattice-equivalent via $f(X) = \lambda f'(\phi(X))$ for some $\lambda \in \Q$, then
$\pp(\mW';z^\lambda) = \pp(\mW;z)$.

By the definition of the M\"obius function, for a weighted lattice $\mW=(\mL,f)$, we have $\pp(\mW; 1) = 0$ and so, unless $\pp(\mW; z)$ is identically zero, $z- 1$ is a factor in $\Z[z]$.
There are some instances for which it is known that $\pp(\mW;z)$ is identically zero, for example, this is the case if $f(a)=0$ for some atom of $a \in \mL$. More generally, as a direct consequence of Lemma \ref{lem:crapo} (\cite[Theorem~1]{crapo1969mobius}), we have the following.

\begin{proposition}\label{prop:charloop}
    Let $\mW=(\mL,f)$ be a weighted lattice and let $A \in \mL$. 
    If $A$ is not a flat of $\mW$ then $\pp(\mW/A;z)$ is identically zero.
\end{proposition}
 
\begin{lemma}\label{lem:chardec}
    Let $\mW=(\mL,f)$ be a weighted lattice. Then
    $\displaystyle \pp(\mW;z) = z^{f(\textbf{1})} -\sum_{B \in \mL:B>\zero} \pp(\mW/B;z)$.
\end{lemma}

\begin{proof}
    By the definition of the characteristic polynomial, for every $B \in \mL$ we have:
    \begin{align*}
    \pp(\mW/B;z) = & \: \sum_{A \in [B,\textbf{1}] } \mu(B,A) z^{f(\one)-f(A)}.
    \end{align*}
    By M\"obius inversion, for every $B \in \mL$ we have:
    $$z^{f(\one)-f(B)} = \sum_{A \in [B,\textbf{1}] }  \pp(\mW/A;z),$$
    and so the result follows by setting $B=\zero$.
\end{proof}

We define the {\em weight enumerator} of an $(\mL,r)$-polymatroid; see \cite{BRS2009} for the matroid case.

\begin{definition}
    Let $\varphi$ be an anti-automorphism of $\mL$.
	We define the {\bf weight enumerator} of the $(\mL,r)$-polymatroid 
    $\M$ to be the list $[A_\M(i;z) : 0 \leq i \leq n]$, where for each $i$ we define
	$$A_\M(i;z):=\sum_{\substack{X \in \mL: \\\h(X)=i}} \pp(\M.X;z)= 
	\sum_{\substack{X \in \mL: \\ \h(X)=i}} \pp(\M/\varphi(X);z) .$$
\end{definition}

\subsection{Rank-Metric Codes}
Rank-metric codes have a natural connection with $q$-polymatroids, as the {\em supports} of their codewords are subspaces. 
We start by briefly recalling some basic notions on rank-metric codes; see \cite{de78,gabidulin1985theory,gorla2021rank}.
For this purpose, we endow the space $\F_{q}^{n \times m}$ with the \textbf{rank distance}, defined by $\mathrm{d}(A,B):=\rk(A-B)$, for all $A,B\in\F_{q}^{n \times m}$. Moreover, we write $U\leq V$ if $U$ is a subspace of $V$.


\begin{definition}
 We say that $C \leq \F_q^{n \times m}$ is an {\bf $\F_q$-linear (rank-metric) code} or a {\bf matrix code} if $C$ is an $\F_q$-subspace of $\F_{q}^{n \times m}$. Its \textbf{minimum distance} is:
 $$\mathrm{d}(C):=\min\{\rk(M) : M \in C, \; M \neq 0\}.$$ We say that $C$ is an $\F_q$-$[n \times m, k,d]$ rank-metric code if it has $\F_q$-dimension $k$ and minimum distance $d$. 
The \textbf{dual code} of $C$ is defined to be $C^{\perp}=\{M\in\F_q^{n\times m}: \Tr(MN^\top)=0 \textnormal{ for all } N\in~C\}$.
For each $i \in [\min(n,m)]$, we define $W_i(C) := |\{ A \in C : \rk(A) = i \}|$. The list $[W_i(C) : i\in~[\min(n,m)]$ is called the {\bf weight distribution} of $C$.
\end{definition}

Let $X \in \F_q^{n \times m}$ and let $U\leq \F_q^n$. We say that $U$ is the \textbf{support}  of $X$ if $\col(X)=U$. Let $C$ be an $\F_q$-$[n\times m,k]$ rank-metric code. We say that $U$ is a {\bf support} of $C$ if there exists some $X \in C$ with support $U$.
We also define the notion of support for a code. 
The \textbf{support} of $C$ is defined to be the $\F_q$-subspace of $\F_q^n$ given by
$$\mathrm{supp}(C)=\sum_{M\in C}\colsp(M),$$
where the sum denotes the sum of vector subspaces. The code $C$ is said to be \textbf{non-degenerate} if $\mathrm{supp}(C)=\F_q^n$.

In the same way one can define the supports of a matrix or of a code in terms of the row-space, however this is not necessary for the scope of the paper.
See \cite{gorla2021rank} for a detailed analysis of the various definitions of rank-support proposed in the literature.

  \begin{definition}
	Let $\Gamma$ be a basis of $\F_{q^m}$ over $\F_q$.
	For each $x \in \F_{q^m}^n$, we write $\Gamma(x)$ to denote the $n \times m$ matrix over $\F_q$ whose $i$th row is the coordinate vector of
	the $i$th coefficient of $x$ with respect to the basis $\Gamma$.
	The \textbf{rank} of $x$, denoted by $\rk(x)$, is the rank of the matrix $\Gamma(x)$. 
\end{definition}
It is easy to see that the rank of $x$ is independent of the choice of the basis $\Gamma$.

\begin{definition}
An \textbf{$\F_{q^m}$-linear (rank-metric) code} or a \textbf{vector code} $C$ is an $\F_{q^m}$-subspace of $\F_{q^m}^n$.
Its \textbf{minimum rank distance} is
$$\mathrm{d}(C)=\min\{\rk(x) : x\in C, \; x\ne 0\}.$$
We say that $C$ is an $\F_{q^m}$-$[n,k,d]$ rank-metric code if it has ${\F_{q^m}}$-dimension $k$ and minimum rank distance $d$. If $d$ is not known we simply write that $C$ is an $\F_{q^m}$-$[n,k]$ rank-metric code. 
A \textbf{generator matrix} of $C$ is a matrix $\smash{G \in \F_{q^m}^{k \times n}}$ whose rows generate $C$ as an $\F_{q^m}$-linear space.
The code $C^\perp$ denotes the \textbf{dual code} of $C$ with respect to the standard dot product on $\F_{q^m}^n$.  
\begin{notation}
    By abuse of notation, we write $C^\perp$ to denote the dual code of $C$ where: 
   	\begin{enumerate}
   	    \item $C \leq \F_q^{n \times m}$ and $C^\perp=\{X \in \F_q^{n \times m}: \mathrm{Tr}(XY^T)=0 \:\forall \:Y \in C\}$;
	\item $C \leq \F_{q^m}^n$ and $C^\perp=\{x \in \F_{q^m}^n :x \cdot y := \sum_{i=1}^n x_i y_i=0 \:\forall \: y \in C \}$.
   	\end{enumerate}
   	\end{notation}
Let $U \leq \Fq^n$ and let $x \in C$. We say that $U$ is a {\bf support} of $x$ if $U$ is the column space of $\Gamma(x)$, where $\Gamma$ is a basis for the field extension $\F_{q^m}/\F_q$ and we write $\supp(x) = U$. In fact, the support does not depend on the choice of $\Gamma$. The \textbf{ support} of $C$ is defined to be
$$ \supp(C)=\sum_{x\in C}\supp(x).$$
We say that $C$ is \textbf{non-degenerate} if $\supp(C)=\F_q^n$.
\end{definition}

In \cite{alfarano2021linear} the following characterization of the non-degeneracy property of vector codes has been provided, which we recall in order to use it in sections \ref{sec:crit_theorem} and \ref{sec:lower_bound}.

\begin{proposition}\cite[Proposition 3.2]{alfarano2021linear}\label{prop:matrix_nondeg}
Let $C$ be an $\mathbb{F}_{q^m}$-$[n,k]$ rank-metric code. $C$ is non-degenerate if and only if the $\F_q$-span of the columns of any generator matrix of $C$ has $\Fq$-dimension~$n$.
\end{proposition}

We recall the notion of equivalence and the puncturing operation for both matrix and vector rank-metric codes. 
For a rank-metric code $C\leq\F_q^{n\times m}$ and matrices $A\in \F_q^{n\times n}$, $B\in\F_q^{m\times m}$ we define:
\begin{align*}
     ACB:=\{AMB : M \in C\} \text{ and } C^T:= \{ M^T : M \in C\}.
\end{align*} 
If $B=I$ we simply write $AC$ in the above and if $A=I$ we write $CB$.
If $n \neq m$, then a pair of rank-metric codes $C_1$ and $C_2$ are called \textbf{equivalent} if there exist $A\in\GL(n,\F_q)$ and $B\in\GL(m,\F_q)$ such that $C_2=AC_1B$. If $n = m$, then the transpose operation for matrices also preserves rank. For this reason it is often included in the definition of equivalence. That is, if $n=m$ then $C_2$ and $C_1$ are called equivalent if $C_2 = AC_1B$ or $C_2^T = AC_1B$ for some invertible matrices $A,B$. If $C_1,C_2$ are two $\F_{q^m}$-$[n,k]$ codes, then they are said to be equivalent if there is a matrix $A\in\GL(n,\F_q)$ such that $C_2=C_1A$. 

Let $M\in\F_q^{n\times m}$. For any $J\subseteq [n]$, satisfying $0 < |J| < n$, we denote by $M_J\in\F_q^{(n-|J|)\times m}$ the submatrix obtained from $M$ by deleting the rows indexed by $J$. 

\begin{definition}
    Let $A\in\GL_n(q)$ and $J\subseteq [n]$ be such that $0<|J|<n$. We define the \textbf{punctured} code of $C$ with respect to $A$ and $J$ to be
    $$\Pi(C,A,J):=\{(AM)_J : M\in C\}\leq \F_q^{(n-|J|)\times m}.$$
\end{definition}

    \begin{definition}
    For arbitrary $W \leq \Fq^n$, we define 
    $C(W):=\{ M \in C : \colsp(M) \leq W\} \leq C$
    to be the {\bf shortened subcode} of $C$ with respect to $W$.
    \end{definition}
    For a coordinate-free description, we may consider $C$ to be a set
    of bilinear maps $b: \Fq^n \times \Fq^m \longrightarrow \Fq$. 
    Then $C$ is realised as an $\Fq$-$[n\times m,k]$ matrix code as $\{(b(u,v): u \in \B_n, v \in \B_m): b \in C\}$ for some bases 
    $\B_n,\B_m$ of $\Fq^n$ and $\Fq^m$, respectively. 
    Given $\Fq$-vector spaces $U$ and $V$, denote the set of bilinear maps $b :U \times V \longrightarrow \Fq$ by $\B(U,V)$.
    For each $b \in \B(U,V), \; y \in U$, let $f_{b,y}$ be
    the $\Fq$-linear map $f_{b,y} :U \longrightarrow \Fq$ defined by
    $f_{b,y}(x) = b(x,y)$ for all $x \in U$. We write
    $\bar{f}_{b,y}$ to denote the unique element of $U$ defined
    by $f_{b,y}(x) = x\cdot \bar{f}_{b,y}$ for all $x \in U$.
    Then for any $b \in \B(U,V)$, the support of $b$ is given by: 
    \[
       \supp(b):=\{\bar{f}_{b,y} : y \in V \} \leq U.
    \]
    If $C \leq \B(U,V)$, then $\displaystyle \supp(C) := \sum_{b \in C} \supp(b)$.
    \begin{definition}
        Let $U$ and $V$ be $\Fq$-vector spaces and let $C \leq \B(U,V)$.
        For each $T\leq U$ and $b \in \B(U,V)$, we write $b_{T} \in \B(T,V)$ to denote the restriction 
        of $b$ to $T \times V$, that is, $b_{T}(x,y) =b(x,y)$ for all
        $(x,y) \in T \times V$.
      The {\bf punctured code} of $C$ with respect to $T$ is defined to be:
    \[
    C|_T:=\{ b_{T} : b \in C\}.
    \]
    The {\bf shortened subcode} of $C$ with respect to $T$ is defined to be:
    \[
      C(T):=\{b \in C : \supp(b) \leq T\}.
    \]
    \end{definition}
    If $A \in \Fq^{t \times n}$ has rank $t$ and $C$ is an $\Fq$-$[n\times m,k]$ rank-metric code, then $AC:=\{AX : X \in C\}\leq \F^{t \times n}$ is a matrix code puncturing of $C$.
    Moreover, if $\rowsp(A)=T$, then $AC$ is the matrix code associated with $\hat{C}|_T \leq \B(T,\Fq^m)$ with respect to the basis of $T$ given by the rows of $A$, where $\hat{C}$ is the unique space of bilinear forms in $\B(\Fq^n,\Fq^m)$ whose matrix representations are the elements of $C$.
    \begin{lemma}\label{lem:punker}
        Let $U,V$ be $\Fq$-vector spaces of finite dimension and let $C\leq \B(U,V)$. 
        Let $T \leq U$. Then $C|_T$ and $C/C(T^\perp)$ are isomorphic.
    \end{lemma}
    \begin{proof}
        Let $\phi_T:C \longrightarrow C|_T$ be the $\Fq$-epimorphism defined by $\phi|_T(b)=b_T$ for all $b \in C$.
        Then 
        \begin{eqnarray*}
            \ker \phi_T &=& \{ b \in C : b_T=0\}\\
            & = & \{b \in C : f_{b,y}(x) = 0 \:\forall\: x \in T, y \in V\}\\
            & = & \{b \in C : \bar{f}_{b,y} \in T^\perp \:\forall\: y \in V\}\\ 
            &=& \{b \in C : \supp(b) \leq T^\perp \}\\
            &=& C(T^\perp),
        \end{eqnarray*}
        from which the result follows.
    \end{proof}
    It follows immediately that if $\rho$ is the rank function of the $q$-polymatroid $\M[C]$ for some $C \in \B(U,V)$, then $\rho(T)=\dim(C)-\dim(C(T^\perp))=\dim(C|_T)$ for all $T \leq U$.

    There is also a notion of a {\em shortening} of $C$ (see \cite[Definition 3.2]{BRSIAM17}); it is an isomorphic puncturing of a shortened subcode. 
   Let $U,V$ be $\Fq$-vector spaces of finite dimension and let $C\leq \B(U,V)$.
   Let $S,T\leq U$ such that $S \oplus T = U$. Then
   $$C(S^\perp)(T^\perp)=\{b \in C : \supp(b) \leq T^\perp \cap S^\perp\} = \{b \in C : \supp(b)=U^\perp\} = \{0\},$$ 
   so by Lemma \ref{lem:punker}, $C(S^\perp)$ and $C(S^\perp)|_T$ are isomorphic. 
    \begin{definition}
        Let $U,V$ be $\Fq$-vector spaces of finite dimension and let $C\leq \B(U,V)$.
        Let $T\leq U$ and let $S \leq U$ such that $S \oplus T = U$.
        We define the {\bf shortened code} of $C$ with respect to $S$ and $T$ to be
        $C(S^\perp)|_T$.
    \end{definition}

\begin{definition}
  Let $C$ be an $\F_{q^m}$-$[n,k]$ rank-metric code and $A\in\GL(n,\Fq)$. Let $J\subseteq [n]$, satisfying $0 < |J| < n$. We define the \textbf{punctured code} of $C$ with respect to $A$ and $J$ by 
  $$ \Pi(C,A,J) :=\{(cA)_{J} : c \in C\}\leq \F_{q^m}^{n-|J|}.$$
\end{definition} 
\begin{remark}\label{rem:pucture+nondegeneracy}
The punctured code of a non-degenerate $\F_{q^m}$-$[n,k]$ code is always non-degenerate, by Proposition \ref{prop:matrix_nondeg} (2). More generally, it is easy to see that the punctured code of a non-degenerate $\F_q$-$[n\times m,k]$ rank-metric code is also non-degenerate, for if $\supp(C)=\Fq^n$ then
$\supp(AC) = \Fq^n$ for any $A \in \GL(n,\F_q)$, so that $\supp(\Pi(C,A,J))=\Fq^{n-|J|}$ for any $J \subseteq [n]$.
\end{remark}

We now recall the relations between the parameters of a rank-metric code.
Indeed, a code cannot have large dimension and minimum distance at the same time. The following result by Delsarte expresses a trade-off between these quantities.

\begin{theorem}[The rank-metric Singleton bound; see \cite{de78}]
\label{thm:slb}
Let $C$ be an $\F_q$-$[n\times m,k]$ rank-metric code with $\mathrm{d}(C) \ge d$. We have
\begin{equation*} \label{singletonlikebound}
    k \le \max(m,n)(\min(m,n)-d+1).
\end{equation*}
\end{theorem}

One of the most studied families of rank-metric codes are those having the maximum possible dimension for a given minimum distance. 

\begin{definition}
Let $C$ be $\F_q$-$[n\times m,k,d]$ rank-metric code. 
$C$ is called a \textbf{maximum rank distance}  (\textbf{MRD}) code
if its parameters attain the bound 
of Theorem~\ref{thm:slb} with equality, i.e.
$$d= \min(n,m)-\frac{k}{\max(n,m)}+1.$$
$C$ is called a quasi-MRD (QMRD) code if $\max(n,m)$ does not divide $k$ and
$$d= \min(n,m)-\left\lceil \frac{k}{\max(n,m)} \right\rceil+1.$$
If $C$ is QMRD and $C^\perp$ is QMRD then $C$ is called a {\bf dually}-QMRD (DQMRD) code.
\end{definition}

MRD codes have several notable properties. Delsarte showed that MRD codes exist for all values of $q, n, m$, and all $d$; see \cite{de78}. In the following lemma we summarize other relevant properties of these codes; see \cite{de78,dela,gorla2021rank}. 

\begin{lemma}\label{MRD properties}
Let $C$ be an $\F_q$-$[n\times m,k,d]$ MRD code or a DQMRD code. 
Then the following properties hold.
\begin{enumerate}
    \item $C$ is non-degenerate. 
    Indeed, a degenerate code can be isometrically embedded into $\F_q^{(n-1)\times m}$, in which case the Singleton bound would be violated.
    \item The weight distribution of $C$ is known and determined by $q,n,m,k$.  
    \item If $C$ is MRD then $C^{\perp}$ is an MRD code with rank distance equal to $\min(m, n) -d+2$.
    \item 
    If $C$ is DQMRD then $C^{\perp}$ has rank distance $d^\perp=\min(m, n) -d+1$.
\end{enumerate}
\end{lemma}

\subsection{Representable {\em q}-Polymatroids}
It is known that an $\mathbb{F}_{q}$-$[n\times m,k]$ rank-metric code induces a $(q,m)$-polymatroid; see \cite{gorla2019rank, shiromoto19}. One way to describe this correspondence is as follows.
	\begin{definition}\label{def:codepoly}
	   Let $m$ be a positive integer and let $C$ be an $\mathbb{F}_{q}$-$[n\times m,k]$ rank-metric code.
	   For each subspace 
	   $U \leq E$, we define
	   $$C_U:=C(U^\perp)=\{A \in C : \colsp(A) \leq U^\perp\} \text{ and }
	   C_{=U}:=\{A \in C: \colsp(A) = U^\perp\}.$$ 
	   Let $\rho: \mL (E) \longrightarrow \N_{\geq 0}$ be defined by
	   	$\displaystyle \rho(U):=k-\dim(C_U).$
	   	Then $(E,\rho)$ is a $(q,m)$-polymatroid \cite[Theorem 5.3]{gorla2019rank} and we denote it by $\M[C]$.	
	\end{definition}

$(q,r)$-polymatroids arising from rank-metric codes are called \textbf{representable}. We point out that in \cite[Theorem 4.9]{GLJ} the authors give examples of $(q,r)$-polymatroids that are not representable.

Each vector rank-metric code determines a $q$-matroid (a $(q,1)$-polymatroid), as follows.
	
	\begin{definition}\label{def:vector_codes}
	   Let 
	   $C$ be an $\mathbb{F}_{q^m}$-$[n,k]$ rank-metric code.
	   For every 
	   $W \leq \Fq^n$, 
	   we define
	   $$C_W:=\{x \in C : \supp(x) \leq W^\perp\} \text{ and }
	   C_{=W}:=\{x \in C: \supp(x) = W^\perp\}.$$ 
	   Let $\rho: \mL (E) \longrightarrow \N_{\geq 0}$ be defined by
	   	$\displaystyle \rho(W):=k-\dim_{\Fqm}(C_W).$
	   Then $(E,\rho)$ is a $q$-matroid \cite[Theorem 24]{JP18} and we also denote it by $\M[C]$. 	
	\end{definition}
    In \cite{degen2024most}, it is shown that most $q$-matroids are not representable, i.e. they do not arise from $\F_{q^m}$-$[n,k]$-rank metric codes.
    
    Note that if Definition \ref{def:codepoly} is applied to a vector rank-metric code to obtain a $(q,m)$-polymatroid, then this $(q,m)$-polymatroid is a scaling (by a factor of $m$) of the $q$-matroid constructed from the same code using Definition \ref{def:vector_codes}.
  
	It is also common to construct a $q$-matroid starting from the generator matrix of an $\mathbb{F}_{q^m}$-$[n,k]$ rank-metric code $C$; see for instance \cite{JP18, gorla2019rank}. Let $G$ be a $k \times n$ generator matrix for $C$ and for every $U\in\mL(E)$, let $A^U$ be a matrix whose columns form a basis of $U$. Then the map 
\begin{equation*}\label{eq:rank1}
    \rho_G:\mL(E) \to \Z, \  U\mapsto \rk(G A^U)
\end{equation*}
is the rank function of a $q$-matroid. However, it is easy to see that $\rho_G$ and the rank function $\rho$ from Definition \ref{def:vector_codes} are the same; see for instance \cite[Corollary 23]{JP18}.

 We make some remarks on code equivalence. If $C$ is a matrix rank-metric code, then $\M[A C B]$ and $\M[C]$ are lattice-equivalent for any $A \in \GL(n, \F_q)$ and $B \in \GL(m,\F_q)$. This means, in particular, that the characteristic polynomials of $\M[A C]$ and $\M[C]$ are identical. If $C$ is a vector rank-metric code, then the multiplication by $A\in\GL(n,\F_q)$ is on the right, but the result stays the same, i.e. the characteristic polynomials of $\M[C A]$ and $\M[C]$ are identical;
 see \cite[Lemma 25]{byrneweighted}.    

We close this preliminary section with two further observations on the connections between the minors of representable $q$-polymatroids and their corresponding rank-metric codes. The fact that the $q$-polymatroids of the puncturings and shortenings of a matrix code $C$ are (up to equivalence) the restrictions and contractions, respectively, of $\M[C]$ was shown in \cite[Theorem 5.5]{GLJ}. 
 \begin{lemma}\label{lem:pun}
    Let $U,V$ be $\Fq$-vector spaces of finite dimension and let $C\leq \B(U,V)$. Let $T \leq U$.
	Then $\M[C]|_T=\M[C|_T]$.
\end{lemma}
\begin{proof}
    Let $\M[C]= (\mL(U),\rho)$. Then $\M[C]|_T=(\mL(T),\rho_T)$, where $\rho_T(S)=\rho(S)$ for all $S \leq T$.
	On the other hand, we have $C|_T = \{b_T : b \in C\}$ and
    $\M[C|_T] = (\mL(T),\tilde{\rho})$, where $\tilde{\rho}(S) = \dim(C|_T|_S) = \dim(C|_S) = \rho(S) = \rho_T(S)$ for all $S\leq T$.    
\end{proof}
\begin{lemma}\label{lem:sho}
Let $U,V$ be $\Fq$-vector spaces of finite dimension and let $C\leq \B(U,V)$. Let $S,T$ be subspaces of $U$ such that $S\oplus T=U$. 
Then $\M[C]/S \cong \M[C(S^\perp)|_T]$.
\end{lemma}
\begin{proof}
   For each $W \leq T$, we have $W \cap S = \{0\}$ and hence we have a lattice-isomorphism 
   $\phi: \mL(T) \longrightarrow \mL([S,U])$ defined by $\phi(W) = S+W$ for all $W \leq T$.
   By Lemma \ref{lem:punker}, $C(S^\perp)|_T \cong C(S^\perp)/C(S^\perp)(T^\perp)=C(S^\perp)/C((S+T)^\perp)=C(S^\perp)$. 
   Let $\M[C] = (\mL(U),\rho)$
   and consider the $q$-polymatroid $\M[C(S^\perp)|_T]=(\mL(T),\rho')$.
   For each $W\leq T$, we have
   \begin{align*}
       \rho'(W) 
                 =& \dim(C(S^\perp)|_T|_W)\\
                 =& \dim(C(S^\perp)|_W)\\
                 =& \dim(C(S^\perp)) - \dim(C(S^\perp)(W^\perp))\\
                 =& \dim(C(S^\perp)) - \dim(C((S+W)^\perp))\\
                 =& \dim(C) - \dim(C((S+W)^\perp))- (\dim(C)-\dim(C(S^\perp)))\\
                 =& \rho(S+W)-\rho(S)\\
                 =& \rho_{[S,U]}(\phi(W)).
   \end{align*}
\end{proof}
Therefore, if $C$ is an $\Fq$-$[n\times m,k]$ matrix code and $T\leq \Fq^n$ has dimension $t$, then $\M[C]|_T$ and $\M[A C]$ are equivalent for any $A \in \Fq^{t \times m}$ such that $\rowsp(A)=T$.
Moreover, $\M[C]/U$ and $\M[A C(U^\perp)]$ are lattice-equivalent if $U\oplus T=\Fq^n$.


\section{Properties of the Characteristic Polynomial}\label{sec:charpoly}

In this section we examine expressions of the characteristic polynomial of a weighted lattice in terms of the characteristic polynomials of its minors. 
Theorem \ref{th:chardecomp} is a generalisation to weighted lattices of a well-known recursive description of the characteristic polynomial of a matroid in terms of contractions and deletions. The interested reader is referred to \cite{kung95,whittle,bryl_oxley} for further reading.

We first give a preliminary result on the characteristic polynomial of a weighted lattice of height $2$.
\begin{lemma}\label{lem:diamond}
    Let $\mL$ have length equal to $2$ and let $e\in \mA$.
    Let $\mW=(\mL,f)$ be a weighted lattice.
    Then
    \[
    \pp(\mW;z) = \pp(\mW|_e;z)(\pp(\mW/e;z)+1) 
    - \sum_{\substack{ a \in \mA:\\a \neq e}} \pp(\mW/a;z).
    \]
\end{lemma}

\begin{proof}
    Let $\mu$ be the M\"obius function of $\mL$.
    Since $\mu(\zero,a)=-1$ for each atom $a$ of $\mL$, $\mu(\zero,\zero)=1$, and $\mu(\zero,\one)=|\mA|-1$, we see that 
    $\displaystyle \pp(\mW;z)=z^{f(\one)} - \sum_{a\in \mA} z^{f(\one)-f(a)} +|\mA|-1$.
    Furthermore, for every atom $a$ of $\mL$ we have:
    $\pp(\mW|_a;z) = z^{f(a)}-1$ and $\pp(\mW/a;z) = z^{f(\one)-f(a)}-1$.
    Therefore,
    \begin{align*}
        \pp(\mW;z)-\pp(\mW|_e;z)+\sum_{\substack{a\in \mA:\\a \neq e}} \pp(\mW/a;z)
        &= z^{f(\one)} - z^{f(e)} -z^{f(\one)-f(e)}+1\\
        &=(z^{f(\one)-f(e)}-1)(z^{f(e)}-1)\\
        &=\pp(\mW/e;z)\pp(\mW|_e).
    \end{align*}
    The result now follows.
\end{proof}

We now show that Lemma \ref{lem:diamond} holds for weighted lattices of any height. We will then use this theorem to deduce several other results on the characteristic polynomial. We remind the reader that $\mL$ is assumed to have finite length.
 

\begin{theorem}\label{th:chardecomp}
    Let $\mL$ be a lower semimodular lattice.
    Let $\mW=(\mL,f)$ be a weighted lattice and
    let $H \in \mH$. Then
    \[
        \pp(\mW;z) = \pp(\mW|_H;z)(\pp(\mW/H;z)+1)- 
        \sum_{b \in \mA: b \nleq H } \pp(\mW/b;z).
    \]
\end{theorem}

\begin{proof}
   We prove the result by induction on the length of $\mL$, noting that the first step holds by 
   Lemma \ref{lem:diamond}. Now suppose that the result holds whenever $\mL$ has length at most $n-1$. Note that since $\mL$ is graded, we have $\h(H) <\h(\mL)$.
    Applying Lemma \ref{lem:chardec} and using the fact that $\pp(\mW/H;z) = z^{f(E)-f(H)}-1$, we have:
    \begin{align*}
        \pp(\mW;z) - \pp(\mW|_H;z)  = & z^{f(E)} - \sum_{\substack{B \in \mL :\\B > \textbf{0}}} \pp(\mW/B;z) -z^{f(H)} +\sum_{\substack{Y \in \mL:\\ \textbf{0}<Y \leq H}} \pp(\mW([Y,H]);z)\\
        =& z^{f(H)}\pp(\mW/H;z)+ \sum_{\substack{Y \in \mL: \\\textbf{0}<Y \leq H}} \pp(\mW([Y,H]);z)
        -\sum_{\substack{B \in \mL :\\ \textbf{0}<B \nleq H}} \pp(\mW/B;z) \\
        -&\sum_{\substack{B \in \mL :\\ \textbf{0}< B \leq H}} \pp(\mW/B;z). 
    \end{align*}
    For $B\leq H$, since $[\zero, B]$ is lower semimodular, we may apply the induction hypothesis to $\pp(\mW/B;z)$ and $H$ and obtain
    $$\pp(\mW/B;z) = \pp(\mW([B,H]);z)(\pp(\mW/H;z)+1)-\sum_{\substack{Y \in \mL :\\ B \lessdot Y \nleq H} } \pp(\mW/Y;z).$$
    Consider the right-most term in the above as we sum over the non-zero elements of $[\zero,H]$.
    If $Y\nleq H$, then since $\mL$ is lower semimodular
    we have that  
    $Y \wedge H \lessdot Y$, from which it follows that $\{B \in \mL : B \lessdot Y, B \leq H \} = \{ Y \wedge H\}$.
    Therefore, we obtain:
    \begin{align*}
        \sum_{\substack{B \in \mL :\\\textbf{0}<B \leq H}}\sum_{\substack{Y \in \mL: \\ B \lessdot Y \nleq H} } \pp(\mW/Y;z)
        &= \sum_{\substack{B \in [\zero,H]}}\sum_{\substack{Y \in \mL :\\ B \lessdot Y \nleq H} } \pp(\mW/Y;z)
        - \sum_{\substack{Y \in \mL :\\ \textbf{0} \lessdot Y \nleq H} } \pp(\mW/Y;z)\\
        &=\sum_{\substack{Y \in \mL :\\\textbf{0} <Y \nleq H}} \pp(\mW/Y;z)
         \sum_{\substack{B \in \mL:\\B \lessdot Y, B \leq H}} 1
        -\sum_{\substack{Y \in \mL : \\\textbf{0} \lessdot Y \nleq H }} \pp(\mW/Y;z)\\
        &=\sum_{\substack{Y \in \mL :\\\textbf{0}< Y \nleq H}} \pp(\mW/Y;z)
        -\sum_{\substack{Y \in \mL : \\\textbf{0} \lessdot Y \nleq H }} \pp(\mW/Y;z).
    \end{align*}
    Therefore, we have 
    \begin{align*}
        \sum_{\substack{B \in \mL :\\\textbf{0} <B \leq H}} \pp(\mW/B;z)=&
        \sum_{\substack{B \in \mL :\\\textbf{0}<B \leq H}} \pp(\mW([B,H]);z) + \sum_{\substack{B \in \mL :\\\textbf{0}< B \leq H}} \pp(\mW([B,H]);z)(\pp(\mW/H;z)) \\
        -& \sum_{\substack{Y \in \mL :\\ \textbf{0} <Y \nleq H}} \pp(\mW/Y;z)
        +\sum_{\substack{Y \in \mL : \\ \textbf{0} \lessdot Y \nleq H }} \pp(\mW/Y;z).
    \end{align*}
    We hence obtain the following identity, using Lemma \ref{lem:chardec}.
    \begin{align*}
        \pp(\mW;z) - \pp(\mW|_H;z)  
        =& \; \pp(\mW/H;z)
         \left(z^{f(H)}- \sum_{\substack{B \in \mL :\\ \textbf{0} <B \leq H}} \pp(\mW([B,H]);z)\right) 
        -\sum_{\substack{Y \in \mL : \\ \textbf{0} \lessdot Y \nleq H }} \pp(\mW/Y;z) \\
        =& \;\pp(\mW/H;z)\pp(\mW|_H;z) -\sum_{\substack{Y \in \mL : \\ \textbf{0} \lessdot Y \nleq H }} \pp(\mW/Y;z). 
    \end{align*}
\end{proof}

The following are immediate consequences of Theorem \ref{th:chardecomp} (with (2) following with the addition of Proposition \ref{prop:charloop}).

\begin{corollary}\label{cor:decomp123}
   Let $\mL$ be a lower semimodular lattice.
   Let $\mW=(\mL,f)$ be a weighted lattice
   and let $H \in \mH$. The following hold.
   \begin{enumerate}
       \item $\displaystyle \pp(\mW;z) = z^{f(\one)-f(H)}\pp(\mW|_H;z) -\sum_{\substack{y \in \mA : \\ y \nleq H }} \pp(\mW/y;z). $
       \item If $H$ is not a flat then $\displaystyle \pp(\mW;z) = \pp(\mW|_H;z) -\sum_{\substack{y \in \mA :\\ y \nleq H }} \pp(\mW/y;z). $
   \end{enumerate}
\end{corollary}

In the case of an $\mL$-polymatroid $\M=(\mL,\rho)$, by the submodularity of $\rho$ we have the following result. Note that if $A \in \mL$ is a complement of $H \in \mH$, then by the lower semimodularity of $\mL$ we must have $\h(A) = 1$.

\begin{corollary}\label{cor:decomp2}
Let $H \in \mH$. Suppose that $H$ has a complement in $\mL$.
    \begin{enumerate}
    \item If a complement of $H$ is a loop of $\M$, then 
       $\displaystyle \pp(\M|_H;z) =\sum_{\substack{y \in \mA : \\y\nleq H }} \pp(\M/y;z). $
    \item Let $\mL$ be modular. If $H$ is a coloop of $\M$ and $e$ is a complement of $H$ in $\mL$, then 
    $$\displaystyle \pp(\M;z) = \pp(\M|_e;z)\pp(\M/e;z) -\sum_{\substack{y \in \mA : \\ y \nleq H , y \neq e}} \pp(\M/y;z). $$
    \end{enumerate}
\end{corollary}

\begin{proof}
     If a complement of $H$ is a loop then $\rho(\one)-\rho(H)=0$, by the submodularity of $\rho$. Moreover, $\pp(\M;z)=0$ by Crapo's result (Proposition \ref{prop:charloop}). This yields (1).
     Suppose $H$ is a coloop of $\M$ and that $e$ is a complement of $H$ in $\mL$. 
     By the submodularity of $\rho$, we have that $\rho(e) = r$ and so $\pp(\M|_e;z)=z^r-1$. 
     Since $\mL$ is modular, we have that $H \wedge (e \vee A) = A$.
     Therefore, by the submodularity of $\rho$, 
     for every $A \in [\zero,H]$ we have that
     $\rho(A \vee e) - \rho(A) = \rho(A \vee e) - \rho(H \wedge (e \vee A)) \geq 
     \rho(H \vee (A \vee e)) - \rho(H) = \rho(\one)-\rho(H)= r=\rho(e).$ 
     On the other hand, again by the submodularity of $\rho$, we have that $\rho(A \vee e)-\rho(e) \leq \rho(A)$ for each $A \in [\zero,H]$. 
     Therefore, $\rho(A \vee e)-\rho(e) =\rho(A) $ for each $A \in [\zero,H]$.
    It follows that $\M/e$ and $\M|_H$ are lattice-equivalent $(\mL,r)$-polymatroids. This yields:
    \begin{align*}
        \pp(\M;z) &\;=\; (\pp(\M|_e;z)+1)\pp(\M/e;z) -\sum_{\substack{y \in \mA:\\ y \nleq H}} \pp(\M/y;z)\\
        &\;=\;\pp(\M|_e;z)\pp(\M/e;z) -\sum_{\substack{y \in \mA : \\ y \nleq H , y \neq e}} \pp(\M/y;z).
    \end{align*}
\end{proof}
\begin{remark}
    If we apply Corollary \ref{cor:decomp123} (1) to a polymatroid $\M$, then we obtain the statement
    \[
        \pp(\M;z) = z^{\rho(E)-\rho(H)}\pp(\M|_H;z)- \pp(\M/e;z),
    \]
    where $e$ is the unique complement of the coatom $H$ in the Boolean lattice. 
    In the case that $\M$ is a matroid, this returns the more familiar looking contraction-deletion identities
    that can be read in several texts; e.g. see \cite[Section 5]{kung95}.
    For example, if $\rho(E)-\rho(H)=0$, then $e$ is not an isthmus of $\M$ and we have
    $\pp(\M;z) = \pp(\M|_H;z)- \pp(\M/e;z)$, while if $\rho(E)-\rho(H)=1$, then $e$ is an isthmus and as in the proof
    of Corollary \ref{cor:decomp2} (2), we have $\pp(\M/e;z)=\pp(\M|_H;z)=\pp(M\backslash e;z)$, which gives $\pp(M;z) = (z-1)\pp(M\backslash e;z) = \pp(\M|_e;z)\pp(M\backslash e;z)$.
\end{remark}
\begin{remark}
    If we apply Corollary \ref{cor:decomp123} (1) to a $q$-matroid $\M$,
    we obtain the corrected statement of \cite[Theorem 5.14]{jany2022projectivization}. 
\end{remark}


\begin{example}\label{ex:qpolylattice}
In Figure \ref{fig:qpoly} (see also, Figure \ref{fig:qpolylattice}), we illustrate a $(2,3)$-polymatroid $\M$ defined on the subspace lattice~$\mL(\F_2^3)$.
The rank function $\rho$ of $\M$ is defined by a weighting on the edges of the Hasse diagram of $\mL(\F_2^3)$. Blue edges have a weighting of 3, black edges have a weighting of 2, red edges have a weighting of 1, and green edges have a weighting of 0. The rank of an element $A$ of $\mL(\F_2^3)$ is the sum of the weights of the edges of a maximal chain in $[\zero,A]$. 
We have:
\begin{gather*}
    \rho(\langle 110\rangle)=\rho(\langle 111\rangle)=2, \\
    \rho(\langle 100\rangle)=\rho(\langle 010\rangle)=\rho(\langle 011\rangle)=\rho(\langle 101\rangle)=\rho(\langle 001\rangle)=3, \\
    \rho(\langle 010,001\rangle)=\rho(\langle 101,010\rangle)=\rho(\langle 110,001\rangle)=4, \\
    \rho(\langle 100,010\rangle)=\rho(\langle 100,011\rangle)=\rho(\langle 100,001\rangle)=\rho(\langle 101,011\rangle)=\rho(\F_2^3)=5.
\end{gather*}

It is straightforward to check that the weighted lattice shown is indeed a $q$-polymatroid and that its characteristic polynomial is
$\pp(M;z)=z^5-2z^3-5z^2+6z=z(z-1)(z^3-z^2-z-6)$.
Note that $\mL(\F_2^3)$ has no coloops.
Let $H = \langle 100,010\rangle$. Then $\rho(\F_2^3)-\rho(H)=0$ and 
one can check that $\pp(M|_H;z)=z^5-z^3-2z^2+2$. Moreover, we have
$\pp(M/\langle 001 \rangle;z)=z^2-2z+1=(z-1)^2$,
$\pp(M/\langle 101 \rangle;z)=z^2-z=z(z-1)$,
$\pp(M/\langle 111 \rangle;z)=z^3-2z+1=(z-1)(z^2+z-1)$,
$\pp(M/\langle 011 \rangle;z)=z^2-z$.
It can be verified that

\begin{align*}
    \pp(M|_H;z) - (\pp(M/\langle 001 \rangle;z)+\pp(M/\langle 101 \rangle;z)+\pp(M/\langle 111 \rangle;z)+\pp(M/\langle 011 \rangle;z))&=\\
    z^5-z^3-2z^2+2 - (z^2-2z+1)-(z^2-z)-(z^3-2z+1)-(z^2-z)&=\\
    z^5-2z^3-z^2+6z &=
    \pp(M;z),
\end{align*}
as predicted by Corollary \ref{cor:decomp123} (1).
\begin{figure}[ht!]
    \centering
    \includegraphics[scale=1]{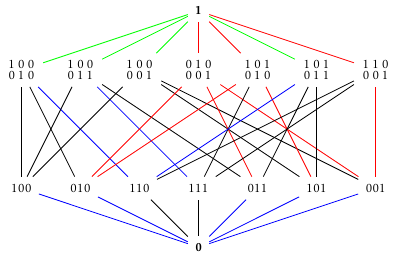}
    \caption{The $(2,3)$-polymatroid from Example \ref{ex:qpolylattice}. \;(\textcolor{blue}{3}, \textcolor{black}{2}, \textcolor{red}{1}, \textcolor{green}{0})}
    \label{fig:qpoly}
\end{figure}

\end{example}


\begin{example}\label{ex:semi}
We remark that the assumption that $\mL$ is lower semimodular is required in order for the statement of Theorem \ref{th:chardecomp} to hold. 
Consider the $\mL_i$-polymatroids, as shown in Figure~\ref{fig:semi}.
We have $\rho_1(a)=1,\rho_1(b)=\rho_1(c)=3, \rho_1(Y)=5,\rho_1(H)=4,\rho_1(\one)=5$, while $\M_2$ is dual to~$\M_1$. Neither $\mL_1$ nor $\mL_2$ is a modular lattice, however, $\mL_1$ is lower semimodular and $\mL_2$ is upper semimodular.
\begin{figure}[ht!]
\centering
\begin{tabular}{l r}
 \includegraphics[]{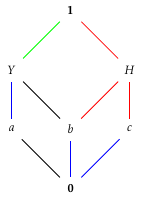}\:\:\:\: \:\:\:\:\:\:\:\: \:\:\:& \:\:\:\:\:\:\:\:\:\:\:\: \:\:\:\includegraphics[]{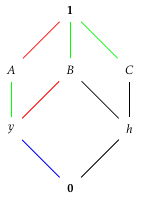}\\
 $\M_1=(\mL_1,\rho_1)$ & $\M_2=(\mL_2,\rho_2)$
\end{tabular}
\caption{Mutually dual lattice polymatroids from Example \ref{ex:semi}. \;(\textcolor{blue}{3}, \textcolor{black}{2}, \textcolor{red}{1}, \textcolor{green}{0})}
\label{fig:semi}
\end{figure}

We have $\pp(\M_1;z)=z^5-z^3-2z^2+z+1$, $\pp(\M_1|_H;z)=z^4-2z+1$, $\pp(\M_1|_Y;z)=z^5-z^3-z^2+1$, and $\pp(\M_1/a;z)=z^3-1$, $\pp(\M_1/b;z)=\pp(\M_1/c;z)=z^2-z$. From this it is easily verified that:
\[
z \pp(\M_1|_H;z) - \pp(\M_1/a;z) =\pp(\M_1|_Y;z) - \pp(\M_1/c;z)=\pp(\M_1;z),
\]
which is consistent with Corollary \ref{cor:decomp123} (1).
Now consider $\M_2$. We have
    $\pp(\M_2;z)=\pp(\M_2|_B;z)=z^4-z^2-z+1$, $\pp(\M_2|_A;z)=z^3-1$, and $\pp(\M_2|_C;z)=z^4-z^2$, 
    $\pp(\M_2/y;z)=0$, $\pp(\M_2/h;z)=z^2-1$.
    However, 
    \[
       z^{\rho_2(\one)-\rho_2(C)}\pp(\M_2|_C;z) -\pp(\M_2/y;z)= z^4-z^2 \neq z^4-z^2-z+1=\pp(\M_2;z).
    \]   
\end{example}


Critical problems for weighted lattices generally involve evaluations of the characteristic polynomial. In the next results, we consider evaluations of the characteristic polynomials of the minors of a weighted lattice. These are all corollaries of Theorem \ref{th:chardecomp}. They will be applied to obtain inequalities involving the critical exponent of a representable $q$-polymatroid in section~\ref{sec:crit_theorem}.

\begin{corollary}\label{cor:charrest}
   Let $\mL$ be a lower semimodular lattice.
   Let $\mW=(\mL,f)$ be a weighted lattice, 
   let $H \in \mH$ and let $g=f(\one)-f(H)$. 
   Let $\theta \in \R$ such that $\theta^g > 0$ and $\pp(\mW/b;\theta)\geq 0$ for every atom $b \leq H$.
   If $\pp(\mW;\theta)>0$ then $\pp(\mW|_H;\theta)>0$.
\end{corollary}

\begin{proof}
   The statement follows as a direct consequence of Corollary \ref{cor:decomp123} (1).
\end{proof}

	\begin{corollary}\label{cor:charcontrest}
    Let $\mL$ be a modular, complemented lattice.
    Let $\mW=(\mL,f)$ be a weighted lattice.
    Suppose there exists $\theta \in \R_{>0} $ such that $\pp(\mW';\theta)\geq 0$ for every minor $\mW'$ of $\mW$.
    Let $T \in \mL$ such that $\pp(\mW/T;\theta)>0$. If $U$ is a complement of $T$, 
    then $\pp(\mW|_U;\theta) >0$. 
\end{corollary}

\begin{proof}
    Let $U$ be a complement of $T$. 
    We claim that $\pp(\mW|_U;\theta) >0$.
    We prove the result by induction on the height of $T$, noting that it holds true vacuously
    for $T=\zero$. 
    Let $a \leq T$, and let $\bar{U} = a \vee U$. Then $\bar{U} \wedge T = (a\vee U) \wedge T = a\vee (T\wedge U)=a$ since $\mL$ is modular. Moreover, 
    $\bar{U} \vee T = \one$, that is, $\bar{U}$ is a complement of $T$ in $[a,\one]$.
    Now $\mW/T$ and $(\mW/a)/T$ are lattice equivalent and so $\pp((\mW/a)/T;\theta) >0$.
    By the induction hypothesis, we have that 
    $0<\pp((\mW/a)|_{\bar{U}};\theta) =\pp(\mW([a,\bar{U}]);\theta) $.
    By the modularity of $\mL$, we have that $a \wedge U = \zero \lessdot a$ implies $U \lessdot a \vee U = \bar{U}$. Therefore, by Theorem \ref{th:chardecomp}, we have,
    \begin{equation}\label{eq:U}
       \theta^{f(\one)-f(U)}\pp(\mW|_U;\theta) =  \pp(\mW|_{\bar{U}};\theta)+
       \sum_{\substack{b \leq \bar{U}: \\b \nleq U}} \pp(\mW([b,\bar{U}]);\theta).
    \end{equation}
    Since $\pp(\mW';\theta)\geq 0$ for every minor $\mW'$ of $\mW$, we have that every term on the right-hand side of (\ref{eq:U}) is non-negative.
    Since $a$ is an atom of $\mL$ such that $a \leq \bar{U}$ and $a \nleq U$, we have that 
    $\pp(\mW([a,\bar{U}]);\theta)$, which is positive, is a summand of (\ref{eq:U}).
    Therefore $\pp(\mW|_U;\theta)>0$ and the result follows.
\end{proof}

\begin{corollary}\label{th:charpuncsho}
    Let $\mL$ be a modular, complemented lattice.
    Let $\mW=(\mL,f)$ be a weighted lattice.
    Suppose there exists $\theta \in \R_{>0} $ such that $\pp(\mW';\theta)\geq 0$ for every minor $\mW'$ of $\mW$.
    Let $U \in \mL$ such that 
    $\pp(\mW|_{U};\theta) >0$. 
    Then there exists a complement $V$ of $U$ in $\mL$ and an element 
    $T \leq V$ such that $\pp(\mW/T;\theta) >0$.
\end{corollary}

\begin{proof}
    We show the result by induction on $n-\h(U)$.
    Clearly, if $U=\one$ then the result holds with $T=V=\zero$.
    Let $H \in \mH$ such that $U \leq H$. 
    Since $\pp((\mW|_H)|_{U};\theta)=\pp(\mW|_{U};\theta) >0$,
    by the induction hypothesis, there exists a complement $V_1\leq H$ of $U$ in $[\zero,H]$ and an element $T_1 \leq V_1$ such that
    $\pp(\mW([T_1,H]);\theta) > 0$. By Theorem \ref{th:chardecomp}, we have
    \[
      \pp(\mW/T_1;\theta) = \theta^{f(E)-f(H)}\pp(\mW([T_1,H]);\theta) 
      - \sum_{\substack{B \in \mL: \\T_1 \lessdot B \nleq H}} \pp(\mW/B;\theta).
    \]
    If either the set $\{B \in \mL : T_1 \lessdot B , B \nleq H\}$ is empty, or if $\pp(\mW/B;\theta)=0$ for each $B$ such that $T_1 \lessdot B,\; B \nleq H$,
    then $\pp(\mW/T;\theta)>0$ for $T=T_1 \leq V_1$ and by the modularity of $\mL$, $V_1$ is contained in an element $V$ that is a complement of $U$ in $\mL$.
    Suppose then that there exists
    a cover $B$ of $T_1$ such that $B \nleq H$ and $\pp(\mW/B;\theta) > 0$.
    Since $\mL$ is modular, any complement of $H$ has height $1$ and any $b \in \mL$ such that $b \nleq H$
    is a complement of $H$. Let $b \leq B$ such that $b \nleq H$.
    Since $T_1 \lessdot B$, we have $B=b \vee T_1.$
    We have that $U \vee V_1 \vee b = H\vee b = \one$.
    Also, it can be checked by the modularity of $\mL$ that $\h((V_1 \vee b) \wedge U)=0$, and so
    $(V_1 \vee b) \wedge U = \zero $. 
    Therefore, $V=V_1 \vee b$ and $T=B=T_1 \vee b$ are the required elements.
\end{proof}


\section{A Critical Theorem for {\em q}-Polymatroids}\label{sec:crit_theorem}
	
	We now consider properties of a $q$-polymatroid arising from an $\F_q$-linear rank-metric code.
	There are several papers outlining properties of rank-metric codes.
	The $q$-polymatroids associated with these structures have been studied in 
	\cite{GLJ,gorla2019rank,shiromoto19}. 
    Throughout this section, we fix $E=\F_q^n$.
   	
	\begin{lemma}[\cite{byrneweighted,gorla2019rank}]\label{lem:repqpoly}
		Let $C$ be an $\mathbb{F}_{q}$-$[n\times m,k]$ rank-metric code.
		The following hold.
		\begin{enumerate}
			\item[(1)] $\M[C^{\perp}] = (\M[C])^*$.
  			\item[(2)] $\pp(\M[C]/U;q)=|C_{=U}|$. 
			\item[(3)] $W_i(C) = A_{\M[C]}(i;q)$ for each $i \in [n]$. 
		\end{enumerate} 
	\end{lemma}
	
\begin{remark}
Note that when $C$ is an $\F_{q^m}$-linear vector code the results of Lemma \ref{lem:repqpoly} hold with~$q^m$ in place of $q$. Indeed, in the Definition~\ref{def:vector_codes}, the rank function of $\M[C]$ is the rank function of the associated $(q,m)$-polymatroid as defined in Definition \ref{def:codepoly}, divided by $m$. Since~$C$ is $\F_{q^m}$-linear, $C_U$ is an $\F_{q^m}$-vector space for each subspace $U$ and so has $\F_q$-dimension which is a multiple of $m$. In particular, we have $\pp(\M[C]/U;q^m)=|C_{=U}|$ for an $\F_{q^m}$-$[n,k]$ code $C$ and subspace $U$. 
	\end{remark}

    Note that Lemma \ref{lem:repqpoly} (2) is an instance of the {Critical Theorem~\cite{CrapoRota}} for representable $q$-polymatroids. 
    We now present a full extension of this result. 
 
 \begin{theorem}[The Critical Theorem]\label{th:critextension}
   Let $t$ be a positive integer, let $C$ be an $\mathbb{F}_{q}$-$[n\times m,k]$ rank-metric code, let $\M=\M[C]$ and let $U \in \mL(E)$.
   Then 
   \[
   |\{(X_1,\ldots,X_t) \in C^t : \sum_{i=1}^t \colsp(X_i) = U \}|  = \pp(\M.U;q^{t}).
   \]
\end{theorem}
\begin{proof}
For each subspace $W\leq \Fq^n$, define
\begin{align*}
    f(W)&:=|\{(X_1,\ldots,X_t) \in C^t: \sum_{i=1}^t \colsp(X_i) = W^\perp \}|,\\
    g(W)&:=|\{(X_1,\ldots,X_t) \in C^t: \sum_{i=1}^t \colsp(X_i) \leq W^\perp \}|.
\end{align*}
We have that $\displaystyle g(W)=\sum_{V \in [W,E]}f(V)$ for all $W \in \mL(E)$, and
\begin{align*}
    g(W) = |\{(X_1,\ldots,X_t) \in C^t:  \colsp(X_i) \leq W^\perp\; \forall \; i \in[t]\}| = |C_W|^t.
\end{align*}
Let $\mu$ be the M\"obius function of $\mL(E)$. Applying M\"obius inversion, we get:
$$ f(W)=\sum_{V \in [W,E]}\mu(W,V)g(V) 
= \sum_{V \in [W,E]}\mu(W,V)|C_V|^t 
= \sum_{V \in [W,E]}\mu(W,V)q^{t(k-\rho(V))} =\pp(\M/W;q^{t}).$$
Now set $U=W^\perp$ to get that $f(U^\perp) =\pp(\M.U;q^{t})$.
\end{proof}

We remark that in~\cite{jany2022projectivization}, a proof of the Critical Theorem for the special case of representable $q$-matroids was given using the concept of the projectivization matroid.
	
The critical problem was originally formulated by Crapo and Rota as the problem of finding the minimum number of hyperplanes that distinguish a set of points in a projective space. We now consider a $q$-analogue of this interpretation. We remark that the problem admits different formulations for other combinatorial structures; see \cite{kung95} for a detailed overview.

For any bilinear form $b:\F_q^n \times \F_q^m\longrightarrow \F_q$, we let $\lker b$
denote the left kernel of $b$.

\begin{definition}
Let $U \leq \F_q^n$. Let $B=(b_1,\ldots,b_t)$ be a list of bilinear forms $b_i \in \B(\F_q^n,\F_q^m)$. 
We say that $B$ {\bf distinguishes} the space $U$ if
$\bigcap\limits_{i=1}^t \lker b_i\leq U^\perp$. 
\end{definition}

Now let $C$ be an $\Fq$-$[n \times m,k]$ rank-metric code and let $X \in C$. 
Clearly, $\lker b = \colsp(X)^\perp$, where $b\in\B(\Fq^n,\Fq^m)$ is the bilinear form represented by $X$. 
We say that a linear map $b$ on $\Fq^n$ is {\bf induced} by $C$ if $b$ arises in this way from a codeword $X$ of $C$. More generally, we say that a $t$-tuple $B = (b_1,\ldots,b_t)$ of bilinear forms is induced by $C$ if each $b_i$ is induced by~$C$. 
The following is a $q$-analogue of \cite[Chapter 16, Theorem 1]{CrapoRota}.

\begin{corollary}\label{cor:crit}
     Let $C$ be an $\Fq$-$[n \times m,k]$ rank-metric code and let 
     $\M=\M[C]$. Then the number of $t$-tuples of bilinear forms $(b_1,\ldots,b_t)$ induced by $C$ that distinguish $\supp(C)$ is $\pp(\M.\supp(C);q^t)$.
     In particular, if $C$ is non-degenerate then this number is $\pp(\M;q^t)$.
\end{corollary}

\begin{proof}
   By Theorem \ref{th:critextension}, we have that 
   \[
   |\{(X_1,\ldots,X_t) \in C^t: \sum_{i=1}^t \colsp(X_i) = \supp(C) \}|
   = \pp(\M.\supp(C);q^t).
   \]
   Let $(b_1,\ldots,b_t)$ be a $t$-tuple of bilinear forms induced by $C$ that distingishes $\supp(C)$ and for each $i$ let $X_i \in C$ be the matrix representation of $b_i$.
   Clearly, $\sum_{j=1}^t \colsp(X_j) \leq \supp(C)$ and hence 
   \begin{align*}
   \bigcap_{i=1}^t \lker b_i\leq \supp(C)^\perp \implies \sum_{i=1}^t \colsp(X_i) = \supp(C). 
   \end{align*} 
\end{proof}

\begin{definition}
Let $C$ be a non-degenerate $\Fq$-$[n \times m,k]$ rank-metric code and let $\M=\M[C]$.
We denote by $\crit(\M)$ the least positive integer $t$ such that there exists a $t$-tuple of bilinear forms ${B}$ induced by $C$ that distinguishes $E$. 
We call this the {\bf critical exponent} of~$\M$.
\end{definition}

It follows that if $\M$ is a representable $(q,m)$-polymatroid, satisfying $\M=\M[C]$ for some $\Fq$-$[n \times m,k]$ rank-metric code $C$, then $\crit(\M)=\min\{t: \pp(\M;q^{t})>0\}$, if such a $t$ exists in~$[k]$.
This is the minimum dimension
$t$ of a subspace $D$ of $C$ such that $\supp(D) = E$. 
Note that $C$ is degenerate if and only if there exists $x\in \mL(E)$ such that $\supp(C) \leq x^\perp$, which holds if and only if $C_x = C$, i.e. if $x$ is a loop of $\M$. Therefore, for a representable $q$-polymatroid $\M$:
\[
   \crit(\M) = \left\{
              \begin{array}{cl}
                \infty                              & \text{ if } \M \text{ has a loop, }\\
                 \min\{t: \pp(\M;q^{t})>0\}  & \text{ otherwise. }
              \end{array}
              \right.
\]

\begin{example}
Let $C$ be the $\F_2$-$[5\times 3,6,1]$ rank-metric code generated by the matrices
$$ \begin{pmatrix}
1 & 0 & 0 \\
0 & 0 & 0 \\
0 & 0 & 0 \\
0 & 0 & 0 \\
0 & 0 & 0
\end{pmatrix}, \quad  \begin{pmatrix}
0 & 0 & 0 \\
0 & 1 & 0 \\
0 & 0 & 0 \\
0 & 0 & 0 \\
0 & 0 & 0
\end{pmatrix}, \quad  \begin{pmatrix}
0 & 0 & 0 \\
0 & 0 & 0 \\
0 & 0 & 1 \\
0 & 0 & 1 \\
0 & 0 & 0
\end{pmatrix},\quad  \begin{pmatrix}
0 & 1 & 0 \\
0 & 1 & 1 \\
0 & 0 & 0 \\
0 & 0 & 1 \\
1 & 0 & 1
\end{pmatrix}, \quad  \begin{pmatrix}
0 & 1 & 1 \\
1 & 0 & 1 \\
1 & 0 & 1 \\
1 & 1 & 1 \\
0 & 1 & 0
\end{pmatrix}, \quad  \begin{pmatrix}
0 & 1 & 1 \\
1 & 1 & 1 \\
1 & 1 & 1 \\
0 & 1 & 0 \\
1 & 0 & 0
\end{pmatrix}.$$
Let $\M=(\F_2^5, \rho)$ be the $(2,3)$-polymatroid $\M[C]$. We calculate its characteristic polynomial: 
\begin{align*}
    \pp(\M;z) &:= \sum_{X\leq \F_2^5}\mu(0,X)z^{6-\rho(X)} 
    =\sum_{j=1}^5(-1)^j2^{\binom{j}{2}}\sum_{\substack{{X\leq \F_2^5}:\\\dim(X)=j}}z^{6-\rho(X)}.
\end{align*} 
With the aid of the algebra package \textsc{magma} \cite{magma}, we obtain:
$$\pp(\M;z)= z^6-4z^4-25z^3+44z^2+40z-56.$$
We have that $\pp(\M;1)=\pp(\M;2)=0$. But $\pp(\M;2^2)= 2280>0$, hence $\crit(\M)=2$. 
Indeed, consider
$$X_1= \begin{pmatrix}
0 & 1 & 0 \\
0 & 1 & 1 \\
0 & 0 & 0 \\
0 & 0 & 1 \\
1 & 0 & 1
\end{pmatrix}  \quad  \textnormal{ and }   \quad  X_2=\begin{pmatrix}
0 & 1 & 1 \\
1 & 1 & 1 \\
1 & 1 & 1 \\
0 & 1 & 0 \\
1 & 0 & 0
\end{pmatrix}.$$
Then $\colsp(X_1)=\langle e_5, e_1+e_2, e_2+e_4+e_5 \rangle$, $\colsp(X_2)=\langle e_4, e_1+e_2+e_3, e_2+e_3+e_5 \rangle$ and it is easy to see that $\colsp(X_1)+\colsp(X_2)=\F_2^5$. 
\end{example}

\begin{remark}
    A version of the Critical Theorem for $q$-polymatroids has been independently shown by Imamura and Shiromoto in \cite{imamura2023}. However, they require that the chosen subspace $U\in \mL(E)$ satisfies $U\cap U^\perp =\zero$.
    This difference is due to a different definition of $\M.U$, which in their case has the following characteristic polynomial: 
    $$\mathbb{P}(\M.U;z)=\sum_{X\in [\zero, U]} \mu(\zero ,X) z^{\rho_{[\zero, U]}(U)-\rho_{[\zero, U]}(X)}.$$
    We also point out that in case $U=E$ the two versions coincide. In particular, this means that if $C$ is a non-degenerate code then $\crit(\M[C])$ can be computed using either definition.
\end{remark}


The following result is a $q$-analogue of \cite[Proposition 6.4.4]{bryl_oxley}. 
\begin{proposition}\label{prop:critineq}
    Let $C$ be an $\Fq$-$[n \times m,k,d]$ non-degenerate rank-metric code and let $\M=\M[C]$.
    Let $T_1, T_2 \in \mL$ such that $T_1 \oplus T_2 = E$. Then
    \[
       \crit(\M|_{T_1}) \leq \crit(\M) \leq \crit(\M|_{T_1}) + \crit(\M|_{T_2}).
    \]
\end{proposition}
\begin{proof}
     Let $\crit(\M) = s$ for some positive integer $s$.
     Then there exists an $s$-dimensional subspace $D \leq C$ whose support is $E$. Clearly, $\supp(D) = E \geq T_1$ and so
     $\crit(\M|_{T_1})\leq s$. 
    Suppose first that $T_1 = E_1=\langle e_1,\dots,e_t \rangle_{\Fq}$ and that $T_2=E_2 = \langle e_{t+1},\dots,e_n\rangle_{\Fq}$.
    Let $\crit(\M|_{E_1})=s_1$ and $\crit(\M|_{E_2})=s_2$ for positive integers $s_1,s_2$, respectively. Then by Lemma \ref{lem:pun}, $\crit(\M[C|_{E_i}])=s_i$ for $i \in \{1,2\}$ and so
    there exist subcodes $D_1,D_2$ of $C$ of dimensions $s_1,s_2$, respectively, such that
    $\supp(D_i|_{E_i}) = E_i$. 
    Therefore, 
    \[
    \supp(D_1) \geq \colsp\left(\left[ \begin{array}{c}
         I_{t} \\
         B_1 
    \end{array} \right]\right)
    \text{ and }
    \supp(D_2) \geq \colsp\left(\left[ \begin{array}{c}
         B_2 \\
         I_{n-t} 
    \end{array} \right]\right),
    \]
    for some $B_1 \in \Fq^{(n-t) \times t}$ and $B_2 \in \Fq^{t \times (n-t)}$.
    Then $\dim(D_1+D_2)\leq s_1+s_2$ and
     \[
      \supp(D_1+D_2)= \supp(D_1) + \supp(D_2) \geq \colsp\left(\left[ \begin{array}{c}
         I_{t} \\
         B_1 
    \end{array} \right]\right)+
    \colsp\left(\left[ \begin{array}{c}
         B_2 \\
         I_{n-t} 
    \end{array} \right]\right)
    =E.
     \]
     Therefore, $\crit(\M) \leq s_1 +s_2$. 
     Now suppose that $T_1$ and $T_2$ are arbitrary subspaces of $E$ of dimensions $t$ and $n-t$, respectively, such that $E=T_1\oplus T_2$.
     Let $A \in \GL(n,\Fq)$ be such that $T_i = \{ xA : x \in E_i\}$ for $i \in \{1,2\}$.
     Then $C|_{T_i} = (A C)|_{E_i}$ and so 
     $\crit(\M) = \crit(\M[A C]) \leq  \crit(\M[A C]|_{E_1}) + \crit(\M[A C]|_{E_2}) = \crit(\M|_{T_1}) + \crit(\M|_{T_2}).$
\end{proof}


    We remark that the left-hand inequality of Proposition~\ref{prop:critineq} is also a consequence of Corollary~\ref{cor:charrest}.
    We have the following statements, which may be deduced from Corollaries \ref{cor:charcontrest} and \ref{th:charpuncsho}. Corollary \ref{cor:extasano} extends \cite{asano}; see also \cite[Section 6.4]{bryl_oxley}.
    
    \begin{corollary}\label{cor:critminors}
       Let $C$ be an $\Fq$-$[n \times m,k]$ non-degenerate rank-metric code, let $\M=\M[C]$
       and let $U \in \mL(E)$. The following hold.
       \begin{enumerate}
           \item There exists $W \in \mL(E)$ that is contained in a complement of $U$, such that 
                 $\crit(\M/W) \leq \crit(\M|_U)$.
           \item For any complement $T\in \mL(E)$ of $U$ we have that $\crit(\M|_T) \leq \crit(\M/U)$.
       \end{enumerate}
    \end{corollary}

    \begin{proof}
        Since $\pp(\M([X,Y]);q^t)$ is the cardinality of a set for every interval $[X,Y] \subseteq \mL(E)$, it is non-negative for every positive integer $t$.
        Suppose that $\crit(\M|_U)=t$ for some positive integer $t$.
        From Corollary \ref{th:charpuncsho}, there exists $W$ contained in a complement of $U$ such that $\pp(\M/W;q^t)>0$ and so $\crit(\M/W)\leq t = \crit(\M|_U)$. This proves (1).
        We show that (2) holds.
        Let $T$ be a complement of $U$ and suppose that $\crit(\M/U)=t$ for some positive integer~$t$. Then $\pp(\M/U;q^t) >0$ and so from Corollary \ref{cor:charcontrest}, we have that $\pp(\M|_T;q^t) >0$. Therefore,
        $\crit(\M|_T) \leq t = \crit(\M/U)$.
    \end{proof}
\begin{corollary}\label{cor:extasano}
 Let $C$ be an $\Fq$-$[n \times m,k]$ non-degenerate rank-metric code, let $\M=\M[C]$
and let $U \in \mL(E)$.
   Let $\ell$ be a positive integer. 
   The following are equivalent.
   \begin{enumerate}
   \item $U$ is minimal in the set $\mS_1:=\{ S \in \mL(E) : \crit(\M|_{S'}) \leq \ell\: \forall \:S' \:\text{s.t.} \;S\oplus S' = E\}$.
   \item $U$ is minimal in the set $\mS_2:=\{ S \in \mL(E) : \crit(\M/S) \leq \ell\}$.
   \end{enumerate}
   In particular, $\mS_2 \subseteq \mS_1$ and the minimal elements of $\mS_1$ coincide with the minimal elements of~$\mS_2$.
\end{corollary}

\begin{proof}
   Clearly, $\mS_2 \subseteq \mS_1$ by Corollary \ref{cor:critminors} (2). 
   Let $U$ be minimal in $\mS_1$. 
   By Corollary \ref{cor:critminors} (1), there exists 
   $S \leq U$ such that $\crit(\M/S) \leq \ell$. Again, by Corollary \ref{cor:critminors} (2), we have $ \crit(\M|_{S'}) \leq \ell$ for every complement
   $S'$ of $S$, i.e. $S \in \mS_1$. By the minimality of $U$, we have $U=S$. Therefore, (1) implies (2).
   
   Let $U$ be minimal in $\mS_2 \subseteq \mS_1$.  
   Let $T \leq U$ be minimal in $\mS_1$.
   By the preceding argument, $T$ is minimal in $\mS_2$, from which it follows that $U=T$. This shows that (2) implies (1).
\end{proof}


We close this section with some observations on the case for which $\M$ is a  $q$-matroid: the specialisation of Corollary \ref{cor:crit}, to the case of $\fqm$-linear codes gives the Critical Theorem for representable $q$-matroids; see also \cite[Theorem 6.20]{jany2022projectivization}. 

	\begin{corollary}
	   Let $C$ be an $\mathbb{F}_{q^m}$-$[n,k,d]$ rank-metric code and let $t$ be a positive integer. Let $U\in \mL(E)$. 
	   Then the number of $t$-tuples $(x_1,\ldots,x_t) \in C^t$ such that
	   $\supp(\langle x_1,\ldots,x_t\rangle) = U$ is $\pp(\M.U;q^{mt})$.
	\end{corollary}

There is a geometric description of the Critical Theorem for representable $q$-matroids that aligns more closely with the matroid case. 
We start by recalling the geometric structure of vector rank-metric codes and their connection with the theory of $q$-systems; see \cite{alfarano2021linear,randrianarisoa2020geometric} for a more detailed treatment.

\begin{definition}
An $[n,k,d]_{q^m/q}$ \textbf{system} $\mU$ is an $\F_q$-subspace of $\F_{q^m}^k$ of dimension $n$, such that
$ \langle \mU \rangle_{\F_{q^m}}=\F_{q^m}^k$ and
$$ d=n-\max\left\{\dim_{\F_q}(\mU\cap H) : H \textnormal{ is an $\F_{q^m}$-hyperplane of }\F_{q^m}^k\right\}.$$
\end{definition}

\begin{theorem}\cite{alfarano2021linear, randrianarisoa2020geometric, sheekey2019scatterd} \label{th:connection}
Let $C$ be a non-degenerate $\F_{q^m}$-$[n,k,d]$ rank-metric code and let $G$ be a generator matrix for $C$.
Let $\mU \leq \F_{q^m}^k$ be the $\F_q$-span of the columns of $G$.
The rank weight of an element $x G \in C$, with $x \in \F_{q^m}^k$ is
$\rk (x G) = n - \dim_{\Fq}(\U \cap \langle x\rangle ^{\perp})$.
\end{theorem}

The $\Fq$-span of the columns of a generator matrix $G$ of a non-degenerate $\F_{q^m}$-$[n,k]$ rank-metric code $C$ is called a $q$-\textbf{system associated to $C$}. There is a bijection between the equivalence classes of non-degenerate $\F_{q^m}$-$[n,k,d]$ rank-metric codes and equivalence classes of $[n,k,d]_{q^m/q}$ systems; the interested reader is referred to \cite{alfarano2021linear, randrianarisoa2020geometric}.

The following result follows easily from Theorem \ref{th:critextension} and Corollary \ref{cor:crit}.

\begin{proposition}\label{prop:geom_crit}
   Let $C$ be a non-degenerate $\F_{q^m}$-$[n,k]$ rank-metric code, let $\mU$ be a $[n,k]_{q^m/q}$ system associated with $C$ and let $\M=\M[C]$. Then
   $$ \crit(\M) = \min\{r\in \N : \exists \;  \mbox{ $\F_{q^m}$-hyperplanes $H_1,\ldots, H_r$ s.t. } 
        \mU \cap H_1\cap\ldots\cap H_r = 0\}.$$
\end{proposition}

\begin{remark}
    Although the definition of $q$-system $\mU$ depends on the choice of generator matrix~$G$, the critical exponent is an invariant of the equivalence class of $\mU$.
\end{remark}

\section{Critical Exponents of Some Rank-Metric Codes}\label{sec:lower_bound}
In this section we compute the critical exponents of $q$-polymatroids arising from some families of $\F_q$-$[n\times m, k]$ rank-metric codes. In the case of a non-degenerate $\Fqm$-$[n,k]$ vector rank-metric code, we show that its corresponding $q$-matroid has critical exponent equal to $\lceil \frac{n}{m} \rceil$. However, determining the critical exponent of an arbitrary representable $q$-polymatroid is not so straightforward.
First we give a lower bound on the critical exponent of the $q$-polymatroid induced by a rank-metric code.

\begin{proposition}\label{prop:LB_critical_exponent}
Let $C$ be  a non-degenerate $\F_q$-$[n\times m,k]$ code and let $\M=\M[C]$. Then
\begin{equation}\label{eq:bound_critical}
    \left\lceil \frac{n}{m} \right\rceil \leq \crit(\M)\leq k.
\end{equation}
\end{proposition}
\begin{proof}
The upper bound follows immediately from the fact that $\dim(C)=k$ and $\crit(\M)\leq \ell$ if $C$ has a non-degenerate subcode of dimension at most $\ell$. 
Now suppose that $\crit(\M)=t$ for some positive integer $t$. Then there exist $X_1,\ldots, X_t\in C$ such that 
$$ \sum_{i=1}^t\colsp(X_i) =\F_q^n.$$
Clearly, the maximum rank that a codeword in $C$ can have is at most $\min(m,n)$, so we get that 
$$ n= \dim_{\F_q}\left(\sum_{i=1}^t\colsp(X_i)\right) \leq mt,$$
and hence we conclude that $t\geq \left\lceil \frac{n}{m}\right\rceil$.
\end{proof}

We now show that the lower bound in Eq.~\eqref{eq:bound_critical} is sharp, i.e. there are codes whose critical exponents meets the bound with equality. 
To this end, if we can guarantee that a code contains enough codewords of rank $\min(n,m)$ with certain distinct supports, then we are done. 
We first consider $\F_{q^m}$-linear codes, making use of the following observation.

\begin{lemma}\cite[Proposition 3.11]{alfarano2021linear}\label{lem:max_rank_weight}
    Let $C$ be a non-degenerate $\F_{q^m}$-$[n,k]$ code, then $C$ contains a codeword of rank equal to $\min(m,n)$.
\end{lemma}

\begin{theorem}\label{thm:n/m_nondeg}
   Let $C$ be  a non-degenerate $\F_{q^m}$-$[n,k]$ code and let $\M=\M[C]$ be the $q$-matroid induced by $C$. Then
   $$ \crit(\M)= \left\lceil \frac{n}{m}\right\rceil.$$
\end{theorem}
\begin{proof}
First of all, observe that $ \crit(\M)\geq \left\lceil \frac{n}{m}\right\rceil$, from Proposition \ref{prop:LB_critical_exponent}.
Write $n=am+b$, with $a,b\in\N_0$ and $0\leq b< m$. We will prove the statement by induction on $a$.
If $a=0$, then $n<m$. In this case, since $C$ is non-degenerate, from Lemma~\ref{lem:max_rank_weight} we immediately conclude that there exists a codeword in $C$ of rank $n$, hence $\crit(\M)=1=\left\lceil \frac{n}{m}\right\rceil$.
Assume that an $\F_{q^m}$-$[n',k]$ non-degenerate rank-metric code with length $n'=a'm+b'$, and $a'<a$, corresponds to a $q$-matroid with critical exponent  $\left\lceil \frac{n'}{m}\right\rceil$.
Since $C$ is non-degenerate, from Lemma \ref{lem:max_rank_weight} there exists a codeword $c\in C$, with rank weight equal to $m$. 
Hence there exists $A\in\GL(n,\F_q)$ such that $x=(x_1,\ldots,x_m,0,\ldots,0)\in CA = \{ xA: x \in C\}$, where $\dim(\supp(cA))=m$. 
It follows that $\langle e_1,\ldots,e_m\rangle_{\Fq} = \supp(x)$, where $e_i$ denotes the $i$-th standard basis vector of $\Fq^n$, for each $i \in [n]$.

\noindent
Let $C_1=\Pi(C,A,[m])\leq \F_{q^m}^{n-m}$, i.e. $C_1$ is the $\fqm$-linear code of length $n-m$ obtained from $CA$ by deleting the first $m$ coordinates of every codeword of $CA$. 
As observed in Remark \ref{rem:pucture+nondegeneracy}, $C_1$ is a non-degenerate code of length $n^\prime=n-m=(a-1)m+b$ and so, by the induction hypothesis, the critical exponent of $\M[C_1]$ is $\left\lceil \frac{n-m}{m}\right\rceil=a$. Hence, there exist some $u^{(1)},\ldots,u^{(a)} \in C_1$ such that $\sum_{i=1}^a\supp(u^{(i)})=
\F_{q}^{n-m}$. Let $x^{(1)},\ldots,x^{(a)} \in CA$ be the codewords in $CA$ corresponding to $u^{(1)},\ldots,u^{(a)}$. In other words,  $u^{(i)}$ is obtained by deleting the first $m$ coordinates of $x^{(i)}$, for every $i=1,\ldots, a$. It is not hard to see that $\langle (v^{(1)},0,\ldots ,0)+e_{m+1},\ldots,(v^{(n-m)},0,\ldots ,0)+e_{n} \rangle_{\Fq} \leq \sum_{i=1}^a\supp(x^{(i)})$, for some $v^{(i)} \in \Fq^m$. 
It follows that $\Fq^{n} = \supp(x) + \sum_{i=1}^a\supp(x^{(i)})$ and so $\crit(\M[CA]) = a+1 = \lceil \frac{n}{m} \rceil$. Since the critical exponent is an invariant of code equivalence, the result follows.
\end{proof}

\begin{remark}
    An alternative proof of Theorem \ref{thm:n/m_nondeg} can be given using the equivalent description of a non-degenerate $\F_{q^m}$-$[n,k]$ rank-metric code as a $q$-system and \cite[Lemma 4.6]{polverino2023divisible}. 
\end{remark}

In general, a non-degenerate $\F_q$-$[n\times m,k]$ rank-metric code does not necessarily contain a codeword of rank $\min(n,m)$; see Example \ref{ex:3x3}, for instance. The next result shows that in some cases we can ensure the existence of such a codeword. 
We will use the fact that for any $\Fq$-$[n \times m,k]$ matrix code $C$, with $m \leq n$, if $k=t \cdot n $ and $\max\{\rk(X) : X \in C\}\leq t$, then $C = U \otimes \Fq^m$ for some subspace $U$ of dimension equal to $t$; see \cite[Theorem~3]{meshulam1985maximal}.

\begin{proposition}\label{lem:opt_anticode}
Let $\ell\geq 2$ and let $A\in\GL(\ell m,\F_q)$. Let $C$ be an  $\F_q$-$[\ell m \times m, m(\ell-1)(m-1)]$ non-degenerate rank-metric code. Let $I\subseteq[\ell m]$ be a set of indices of size $m$. If the punctured code $\Pi(C,A,I)$ has dimension equal to $m(\ell-1)(m-1)$, then $\Pi(C,A,I)$ contains a codeword of rank~$m$. In particular, $C$ contains a codeword of rank~$m$.
\end{proposition}
\begin{proof}
    Let $I\subseteq[\ell m]$ be a set of size $m$. Let $\Pi(C,A,I)\leq \F_q^{m(\ell-1)\times m}$ be the code obtained by deleting the $m$ rows of each matrix in $AC$ indexed by $I$. Assume that $\Pi(C,A,I)$ is an $\F_q$-$[m(\ell-1)\times m, m(\ell-1)(m-1)]$ code.
    Towards a contradiction, suppose that every word of $\Pi(C,A,I)$ has rank at most $ m-1$.  
    By hypothesis, we have that
    $\dim(\Pi(C,A,I)) = m(\ell-1)(m-1).$
    Therefore, by \cite[Theorem~3]{meshulam1985maximal} $C$ is degenerate, which is a contradiction to Remark~\ref{rem:pucture+nondegeneracy}.
\end{proof}

Computing the critical exponent of $\M[C]$ for an arbitrary matrix code is difficult. However, with some extra assumptions on the code we can compute it.

We now consider the class of $\F_q$-$[n\times m,k,d]$ 1-\textbf{binomial moment determined} (BMD) rank-metric codes, which are precisely those codes for which
\[
    \min(n,m)-d < d^\perp
\] 
and which was shown to be the disjoint union of the DQMRD codes and the MRD codes \cite{cotardo20}.
We will use the following result from \cite{rav}.
\begin{lemma}\label{lem:cucuperp}
    Let $C$ be an $\F_q$-$[n\times m,k,d]$ code
    and let $U$ be a $u$-dimensional subspace of $\Fq^n$.
    Then 
    \[
       |C_U|= q^{k-mu}|C^{\perp}_{U^\perp}|.
    \]
    In particular,
    \[
    |C_U|= \left\{ \begin{array}{ll}
    1           & \text{ if } n-d     < u, \\
    q^{k-mu}& \text{ if } d^\perp > u.\\    
    \end{array}  \right.
    \]
\end{lemma}

\begin{lemma}\label{thm:supp}
 Let $n>m$ and let $C$ be an $\F_q$-$[n\times m,k,d]$ $1$-BMD code such that
 $n+d \leq 2m$.
 Let $U$ be an $(n-m)$-dimensional subspace of $\Fq^n$. 
 Then $C_{=U} \neq \emptyset$ and hence $U^\perp$ is a support of a codeword of $C$.
\end{lemma}
\begin{proof}
   Since $C$ is $1$-BMD and $n>m$, we have $d^\perp > m-d \geq n-m =\dim(U)$ and
   so
   from Lemma \ref{lem:cucuperp}, we have that $|C_U|=q^{k-m(n-m)}$.
   Furthermore, if $V$ is any subspace of $\Fq^n$ of dimension $n-m+1$, then $|C_V|=q^{k-m(n-m+1)}$.
   Now, 
   \begin{align*}
      C_U= \;&C_{=U} \cup \{ X \in C : \colsp(X) \leq V^\perp, \text{ for some }V \text{ s.t. } U \leq V, \dim(V) = n-m+1\} \\
      = \;& C_{=U} \cup \bigcup_{\substack{V:U \leq V, \\ \dim(V) = n-m+1}} C_V, 
   \end{align*}
   and so 
 	\begin{align*}
 	q^{k-mu}=&|C_U|
 	            \leq |C_{=U}|+\sum_{\substack{V:U \leq V, \\ \dim(V) = n-m+1}}|C_V|
 	            \leq |C_{=U}| + \qbin{m}{1}{q} q^{k-m(n-m+1)}.
 	\end{align*}
 	Therefore,
 	\[
 	  |C_{=U}| \geq q^{k-m(n-m)}\left(1-\qbin{m}{1}{q} q^{-m}\right)>0.
 	\]
 	It follows that $C_{=U}$ is non-empty and hence $C$ has a codeword whose support is $U^\perp$.
\end{proof}

\begin{theorem}\label{prop:criticalMRD}
Let $C$ be an $\F_q$-$[n\times m,k,d]$ MRD code and let $\M=\M[C]$. Then
$$\displaystyle \crit(\M)=\left\lceil\frac{n}{m}\right \rceil,$$ if one of the following conditions hold:
\begin{enumerate}
    \item $n\leq m$,
    \item $m<n\leq 2m-d$,
    \item $k=n=m+1$.
\end{enumerate} 
\end{theorem}

\begin{proof}
Let $n\leq m$. 
We already observed in Lemma \ref{MRD properties} that MRD codes have the same weight distribution. Moreover, they exist for every choice of parameters $q,n,m,d$. In particular, given an $\F_q$-$[n\times m,m(n-d+1),d]$ MRD code, there exists an $\F_{q^m}$-linear MRD code with the same weight distribution. By Lemma \ref{lem:max_rank_weight}, then $C$ has at least one codeword of rank-weight $n$, hence in this case $\crit(\M[C])=1=\left\lceil\frac{n}{m}\right \rceil.$ 

Now consider the case $2m-d \geq n>m$. By Lemma \ref{thm:supp}, $C_{=U^\perp}$ is non-empty for every $U$ of dimension $m$ and hence for each such $U$ there exists $X \in C$ such that $\colsp(X)=U$. 
The result now follows by taking any collection of $\lceil \frac{n}{m} \rceil$ $m$-dimensional subspaces whose vector space sum is $\Fq^n$.

Finally, when $k=n=m+1$, then $C$ is an MRD code with minimum distance $(n-1)$. Moreover, since every codeword of $C$ is an $n\times (n-1)$ matrix, all the nonzero codewords have rank exactly $(n-1)$. If the supports of all the codewords of $C$ are contained in the same hyperplane, then $C$ is degenerate and we get a contradiction. Hence, there are at least two codewords $M_1, M_2$ with rank $n-1$ whose supports are contained in two distinct hyperplanes in $\F_q^n$, then $\crit(\M)=2$,~since $\colsp(M_1)+\colsp(M_2)=\F_q^n. $  
\end{proof}

Note that in general, the critical exponent of the $q$-polymatroid associated with an $\F_q$-$[n\times m,k]$ rank-metric code need not be $\left\lceil \frac{n}{m}\right\rceil$, as next example shows.

\begin{example}\label{ex:3x3}
Let $C$ be the (non-degenerate) $\F_2$-$[3 \times 3,4,2]$ code generated by the following matrices
$$\begin{pmatrix}
1 & 0 & 0 \\
0 & 1 & 0 \\
0 & 0 & 0
\end{pmatrix}, \quad 
\begin{pmatrix}
0 & 0 & 0 \\
0 & 1 & 0 \\
0 & 0 & 1
\end{pmatrix}, \quad 
\begin{pmatrix}
0 & 1 & 1 \\
0 & 0 & 1 \\
0 & 0 & 0
\end{pmatrix}, \quad 
\begin{pmatrix}
0 & 0 & 0 \\
1 & 0 & 0 \\
1 & 1 & 0
\end{pmatrix}.$$
Every nonzero codeword in $C$ has the same rank $2$, hence, the critical exponent of the $q$-polymatroid $\M[C]$ induced by $C$ is at least $2 $ (and indeed is exactly 2), which strictly exceeds $\left\lceil\frac{n}{m}\right\rceil=1$.
\end{example}

In the discussion above, we pointed out that there are several rank-metric codes whose induced $q$-(poly)matroid has critical exponent that meets the lower bound~\eqref{eq:bound_critical} with equality. In particular, we can summarize all the results above in the following corollary.

\begin{corollary}\label{cor:summary}
   Let $C$ be an $\F_q$-$[n\times m,k,d]$ rank-metric code and let $\M=\M[C]$. If one of the following conditions is satisfied, then $\crit(\M)=\left\lceil\frac{n}{m}\right\rceil$. 
   \begin{enumerate}
       \item[(a)] $C$ is $\F_{q^m}$-linear and non-degenerate.
       \item[(b)] $n\leq m$ and $C$ is MRD.
       \item[(c)] $m<n\leq 2m-d$ and $C$ is $1$-BMD.
       \item[(d)] $m=n-1$, $k=n$ and $C$ is MRD.
   \end{enumerate}
\end{corollary}

The conditions of Corollary \ref{cor:summary} are not necessary, as the following example illustrates.

\begin{example}\label{ex:4x3}
Let $C$ be the $\F_2$-$[4\times 2,3,1]$ rank-metric code, generated by the following matrices
\begin{align*}
    X_1:=\begin{pmatrix}
    1 & 0 \\
    0 & 0 \\
    0 & 1 \\
    0 & 1
    \end{pmatrix}, \   X_2:=\begin{pmatrix}
   0 & 1 \\
   1 & 0\\
   1 & 0\\
   0 & 1
    \end{pmatrix},  \   X_3:=\begin{pmatrix}
   1 & 0 \\
   1 & 0 \\
   1 & 1 \\
   0 & 1
    \end{pmatrix}.
\end{align*}
Note that $C$ is linear over $\F_2$, but not over $\F_{4}$, hence condition (a) is not satisfied. Clearly, also (b)--(d) are not satisfied. However, $\crit(\M[C])=2=n/m$, indeed $\colsp(X_1)+\colsp(X_2) = \F_2^4$.
\end{example}

It is known that if $n\leq m$, the $q$-polymatroid induced by an MRD code is the uniform $q$-matroid $\mathcal{U}_{n-d+1,n}(q)$ of rank $n-d+1$, where $d$ is the minimum rank distance of the code. However, this is not true for $n>m$; see \cite{gorla2019rank}.

\begin{corollary}\label{cor:uniformqmatroid}
Let $n \leq m$.
The uniform $q$-matroid $\U_{k,n}(q)$ is representable over $\F_{q^m}$ and has critical exponent $\crit(\U_{k,n}(q))=1$. 
\end{corollary}

We point out the similarity between the critical exponent of the matroid induced by an MDS code and the one of the $q$-matroid induced by an $\F_{q^m}$-linear MRD code. 
In the matroid case, a uniform matroid $U_{k,n}$ that is representable over $\F_q$ has critical exponent $\crit(U_{k,n})\leq 2$.
In particular,
$$\crit(U_{k,n})=\begin{cases}
2 & \textnormal{ if } n=q+1 \textnormal{ and } k=2, \\
1 & \textnormal{ otherwise. }
\end{cases}$$

Simplex rank-metric codes have been recently defined as the natural counterpart of simplex Hamming-metric codes from a geometric point of view; see \cite{alfarano2021linear, randrianarisoa2020geometric}. In these works it has been shown that simplex rank-metric codes are the only non-degenerate one-weight
codes in the rank-metric, just like simplex codes in the Hamming metric, up to repetition. They are formally defined as follows.

\begin{definition}\label{def:simplex}
   Let $k\geq 2$ and $C$ be an $\F_{q^m}$-$[mk,k]$ non-degenerate code. Then $C$ is a one-weight code with minimum distance $m$ and it is called a \bf{simplex rank-metric} code. 
\end{definition}

As an immediate consequence of Proposition \ref{prop:LB_critical_exponent} we have the following result. 

\begin{corollary}\label{cor:critsimplex}
   Let $k\geq 2$, let $C$ be the $\F_{q^m}$-$[mk,k,m]$ simplex rank-metric code and let $\M=\M[C]$ be the $q$-matroid associated to $C$. Then $\crit(\M)=k$.
\end{corollary}

Note that $U_{2,q+1}$ is representable over $\F_q$ as an $\F_q$-$[q+1,2]$ Hamming-metric code, satisfying the classical Singleton bound. This code is also a simplex code and in analogy with Corollary~\ref{cor:critsimplex}, its critical exponent is equal to the dimension of the code.

\section{Generalisations of the Critical Theorem}\label{sec:gen_britz}

In \cite{britz2005extensions}, the Critical Theorem was extended in order to describe, to the widest possible extent, the matroidal properties of a linear code.
If we consider the Critical Theorem (either for representable matroids or $q$-polymatroids), 
what is being counted is the number of $r$-tuples of codewords with a certain property, namely that the support of the code spanned by these codewords is equal to a given element $A$ of the underlying support lattice $\mL$, which is the Boolean lattice in the case of matroids and the subspace lattice in the case of $q$-polymatroids. 
The Critical Theorem shows that this quantity is determined by an evaluation of the characteristic polynomial of the associated ($q$)-(poly)matroid, and in particular is a function of the ranks of the elements in the interval $[A,\one]$. 
In revisiting this topic, in \cite{britz2005extensions} Britz considered the problem of counting more general objects arising from a linear code, and showed that if such {\em structures} exhibit a certain level of invariance, then the number of such structures with support equal to $A$ is also determined by the ranks of the elements in the interval  $[A,\one]$. A similar result can be stated for $q$-polymatroids.

Let $\mC$ be a multiset of elements of a set $S$. 
A structure of order $1$ over $\mC$ is a finite multiset or a finite tuple of elements of $\mC$. A structure of order $2$ over $\mC$ is a finite multiset or finite tuple of structures of order at most $1$ over $\mC$. We hence recursively construct a structure of order $\ell$ over $\mC$ as either a finite tuple or a finite multiset of structures, each of order at most $\ell-1$ over~$\mC$.
The ground set of $\mC$ is the set $\G(\mC):=S$. 
More generally, the ground-set of a structure~$\gamma$ over~$\mC$ is the set $\G(\gamma):=\cup_{\alpha \in \gamma} \G(\alpha)$.
A collection of structures over $\C$ may be defined in terms of a predicate $\X$, in which case we denote this set of structures
by $\X(\C)$.
Several examples of such structures are listed in \cite[Tables 1 and 2]{britz2005extensions}, for the case that $\mC$ is a multiset of elements of a union of linear codes, possibly over different finite fields. 

\begin{notation}
We adopt the following notation throughout this section.
\begin{itemize}
    \item $\mu$ denotes the M\"obius function on $\mL(E)$. 
    \item $k:=\rho(E)$.
    \item $C^{(i)}$ denotes an arbitrary but fixed $\Fq$-$[n \times m_i,k_i]$ rank-metric code, for each $i \in [s]$.
    \item For each $U \in \mL(E)$, $U^\perp$ denotes the image of $U$ under an arbitrary but fixed
    anti-automorphism of $\mL(E)$. 
    \item $\M[C^{(i)}]=(\mL(E),\rho_i)$, i.e., $\rho_i(U):=\dim_{\F_q}(C^{(i)})-\dim_{\F_q}(C^{(i)}_U)$, for each $i \in [s]$.
\end{itemize}
We assume that $\M[C^{(1)}], \ldots, \M[C^{(s)}]$ are all scaling-equivalent to the $(q,m)$-polymatroid $\M=(\mL(E),\rho)$. That is, we assume that 
for each $i \in [s]$, there exists a
positive $\ell_i\in \Q$ such that $\rho(U) = \ell_i\rho_i(U)$ for each $U \in \mL(E)$.
This means in particular that $k = \ell_i k_i$ for each $i \in [s]$.
\end{notation}

We will be concerned with structures over $\mC$ for which 
$\mC=(C^{(1)},\dots,C^{(s)})$.

\begin{definition}
We say that the predicate $\X$ is {\bf invariant} if 
the cardinality of $\X((C^{(1)},\dots,C^{(s)}))$ depends only on $((m_1,\ell_1),\ldots,(m_s,\ell_s))$ and $k$. If $\X$ is invariant, we denote this cardinality by\\ $\Theta(\X;((m_i,\ell_i): i \in [s]);k)$.
If $(m_i,\ell_i) = (N,L)$ for some positive integers $N,L$ for each $i \in [s]$, 
we will use the notation $\Theta(\X;(N,L);k)$.
\end{definition}

For any structure $\gamma$ over the multiset $\mC=(C^{(1)},\ldots,C^{(s)})$ we define the \textbf{support} of $\gamma$ to be 
$\supp(\gamma):=\sum_{X \in \G(\gamma)} \colsp(X)$. In other words, the support of $\gamma$ is the sum of the column spaces of the elements of the ground set of $\gamma$.

\begin{lemma}\label{lem:shosub}
    Let $C$ be an $\Fq$-$[n \times m,k]$ code and let $\M=\M[C]$.
    Let $U\in \mL(E)$.
    Then $M[C_U]=(\mL(E),\rho_U)$, where $\rho_U(V):=\rho(U+V)-\rho(U)$ for all $V \in \mL(E)$.
\end{lemma}

\begin{proof}
    We have $\rho_U(V) = \dim(C_U)-\dim((C_U)_V)$ for all $V\in \mL(E)$ by definition. Since 
    \[
    (C_U)_V = \{X \in C : \colsp(X) \leq U^\perp \cap V^\perp \} = \{X \in C : \colsp(X) \leq (U+V)^\perp \} = C_{U+V},
    \]
    we get
    $\rho_U(V)=\dim(C_U)-\dim(C_{U+V}) = k-\rho(U)-(k-\rho(U+V))= \rho(U+V)-\rho(U).$
\end{proof}

\begin{lemma}\label{lem:inv}
    Let $\X$ be an invariant predicate and let $U \in \mL(E)$. Then
    \[
  |\{ \gamma \in \X(C^{(1)}_U,\ldots,C^{(s)}_U)\}|
  =\Theta(\X;((m_i,\ell_i):i \in [s]);k-\rho(U)).
   \]
\end{lemma}
\begin{proof}
From Lemma \ref{lem:shosub}, for each $i \in [s]$, $\M[C^{(i)}_U]=(\mL(E),\rho_{i,U})$ where
$\rho_{i,U}(V)=\rho_i(U+V)-\rho_i(U)$.
Let $\M_U=(\mL(E),\rho_U)$, be the $q$-polymatroid for which
$\rho_U(V):=\rho(U+V)-\rho(U)$ for all $V \in \mL(E)$.
Then for each $i \in [s]$, $\ell_i \rho_{i,U}(V) =\rho_U(V)$ for all $V \in \mL(E)$ and hence $\M[C^{(i)}_U]$ is scaling-equivalent to $\M_U$ and $\ell_i \rho_{i,U}(E)=\rho_U(E)=k-\rho(U)$.
Since $\X$ is invariant, the result now follows. 
\end{proof}
The following is an extension of the Critical Theorem. We will apply it to counting the number of structures over $\mC$ whose support is equal to a fixed element of $\mL(E)$. 

\begin{theorem}\label{th:gencrit}  
   Let $\X$ be an invariant predicate and let $\mC=(C^{(1)}, \ldots , C^{(s)})$. Then for every $U \leq E$, we have:
   \[
   | \{\gamma \in \X(\mC): \supp(\gamma)=U^\perp \}| = \sum_{V \in [U,E]}\mu(U,V)\Theta(\X;((m_i,\ell_i):i \in [s]);k-\rho(V)).
   \]
\end{theorem}

\begin{proof}
   Let $W\in \mL(E)$.
   By definition, for any $\gamma \in \X(\mC)$ 
   we have that $\supp(\gamma) = \sum_{X \in \G(\gamma)} \colsp(X)$. Therefore, 
   we have that $\supp(\gamma)\leq W^\perp$ if and only if $\colsp(X) \leq W^\perp$
   for each $X \in \G(\gamma)$. In particular, we have that a structure over
   $\mC$ has support contained in $W^\perp$ if and only if it is a structure
   over $(C^{(1)}_W , \ldots , C^{(s)}_W)$. For each $i$, we have that
   $M[C^{(i)}_W] = (\mL(E),\rho_{i,W})$ where $\ell_i\rho_{i,W}(Z)=\rho_W(Z)$ for all $Z \in \mL(E)$.
   Also, $\dim(C^{(i)}_W) = k_i-\rho_i(W)=\ell_i^{-1}(k-\rho(W))$.
   By the invariance of $\X$, from Lemma \ref{lem:inv} we have that
   $|\{ \gamma \in \X(C^{(1)}_W, \ldots, C^{(s)}_W)\}| 
   = \theta(\X;((m_i,\ell_i):i \in [s]);k-\rho(W))$.
   It follows that for all $U \in \mL( E)$, we have: 
   \[
   |\{\gamma \in \X(\mC): \supp(\gamma)\leq U^\perp \}| = \sum_{V \in [U,E]} \Theta(\X;((m_i,\ell_i):i \in [s]);k-\rho(V)).
   \]
   The result now follows by applying M\"obius inversion.
\end{proof}

We obtain other generalisations of the Critical Theorem by appropriate specialisations of Theorem \ref{th:gencrit}. We mention a few here. 

\begin{enumerate}
\item[\textbf{1.}]
We retrieve Theorem \ref{th:critextension} by setting $\X(\mC) = \{ (X_1,\ldots,X_s) : X_i \in C\}$, which is a collection of structures of order $1$ over $C$ for some $\Fq$-$[n \times m,k]$ rank-metric code $C$. Then
\begin{align*}
    \Theta(\X;(m,1);k-\rho(U))&=|\{(X_1,\ldots,X_s): X_i \in C, \colsp(X_i)) \leq U^\perp,i \in [s]\}|
    \\ &=|C_U|^s = q^{s(k-\rho(U))}.
\end{align*}
\item[\textbf{2.}]
If we set $\X(\mC) = \{ (X_1,\ldots,X_s): X_i \in C^{(i)},\; i \in [s]\}$, then we retrieve a $q$-analogue of \cite[Theorem 4.3]{kung95}. 
We have 
$$\Theta(\X;((m_i,\ell_i): i \in [s]);k-\rho(U))=
|\{(X_i: i \in [s]):\colsp(X_i) \leq U^\perp\}|=\prod_{i=1}^s|C^{(i)}_U|.$$
This yields that:
\begin{align*}
   |\{\gamma \in \X(\mC):\supp(\gamma) = U^\perp\}| 
&= \sum_{V\in [U,E]} \mu(U,V) \prod_{i=1}^s|C^{(i)}_V| \\
&=\sum_{V\in [U,E]} \mu(U,V)\prod_{i=1}^s q^{\ell_i^{-1}(k-\rho(V))}.
\end{align*}

\item[\textbf{3.}]
Let $\X(\mC) = \{(D_1,\dots,D_s) : D_i \leq C^{(i)},\dim(D_i)=d_i\; i \in [s] \}$. That is, $\X(\mC)$ is the structure of $s$-tuples of $\Fq$-subcodes of $C^{(i)}$ of dimension $d_i$. Then
$$\Theta(\X;((m_i,\ell_i): i \in [s]);k-\rho(U))=
|\{\gamma \in \X(\mC):\supp(\gamma) \leq U^\perp\}|=\prod_{i=1}^s\qbin{k_i-\rho_i(U)}{d_i}{q},$$
and so
$$|\{\gamma \in \X(\mC):\supp(\gamma) = U^\perp\}|=\sum_{V\in [U,E]} \mu(U,V)\prod_{i=1}^s\qbin{\ell_i^{-1}(k-\rho(V))}{d_i}{q}. $$
This gives a $q$-analogue of \cite[Corollary 7]{britz2005extensions}.
\end{enumerate}

We illustrate how to compute the number of pairs $(X_1, X_2)$ of matrices, such that $X_1\in C^{(1)}$ and $X_2\in C^{(2)}$, with $\colsp(X_1)+\colsp(X_2)$ equal to a fixed space. 
\begin{example}\label{ex:1}
    Let $q=2$ and $C^{(1)}$ be the $\F_2$-$[3\times 3,3]$ with basis
    $$\begin{pmatrix}
        1 & 0 & 0 \\ 
        1 & 0 & 0 \\
        1 & 1 & 0 
    \end{pmatrix}, \qquad \begin{pmatrix}
        0 & 1 & 0 \\
        0 & 1 & 1 \\
        1 & 1 & 0
    \end{pmatrix}, \qquad \begin{pmatrix}
        0 & 0 & 1 \\
        1 & 1 & 0 \\
        0 & 1 & 0 
    \end{pmatrix}.$$
      Let $C^{(2)}$ be the $\F_2$-$[3\times 3,3]$ with basis
    $$\begin{pmatrix}
        1 & 0 & 1 \\ 
        0 & 0 & 1 \\
        0 & 1 & 0 
    \end{pmatrix}, \qquad \begin{pmatrix}
        0 & 1 & 1 \\
        0 & 0 & 1 \\
        0 & 0 & 1
    \end{pmatrix}, \qquad \begin{pmatrix}
        0 & 0 & 0 \\
        0 & 1 & 0 \\
        1 & 1 & 1 
    \end{pmatrix}.$$
    Let $\rho_i$ be the rank function of the $q$-polymatroid arising from $C^{(i)}$, for $i=1,2$ and $k_i$ be its rank.
    Let $U = \langle (1,1,1)\rangle\leq \F_2^3$. Then $U^\perp = \langle (1,0,1), (0,1,1) \rangle$
    and we have the following:
    \begin{align*}
        |\{(X_1,X_2) \; : \; X_i\in C^{(i)}, \; \colsp(X_1)+ \colsp(X_2)= U^\perp\}| &=\sum_{U\leq V\leq \F_2^3}\mu(U,V)\prod_{i=1}^2q^{k_i-\rho_i(V)}\\
        &=2q - \sum_{U< V<\F_2^3}q^{\binom{1}{2}}\prod_{i=1}^2q^{k_i-\rho_i(V)} \\
        &=2q-3 = 1.
    \end{align*}
    Indeed, it is not difficult to see that the only matrix in $C^{(1)}$ whose column space is contained in $U^\perp$ is the zero matrix, while there is a matrix $N=\begin{pmatrix}
        1 & 1 & 0\\
        0 & 1 & 0 \\
        1 & 0 & 0
    \end{pmatrix}\in C^{(2)}$, whose support is equal to $U^\perp$. Hence, the only pair is given by $(0,N)$.
\end{example}
The following is an example of a structure of order $2$.
\begin{example}
    Let $C^{(1)}$ and $C^{(2)}$ be the codes from Example \ref{ex:1}. Let 
    $$\X(\C)=\{((M_1,M_2), (N_1,N_2)) \; : \; M_i\in C^{(1)}, \; N_i\in C^{(2)}, \; i=1,2\}.$$
    That is, $\X(\C)$ is the structure of pairs of \textit{pairs} of codewords of $C^{(1)}$ and $C^{(2)}$.  Let $U = \langle (1,1,1)\rangle\leq \F_2^3$. Then $U^\perp = \langle (1,0,1), (0,1,1) \rangle$
    and we have the following:
    \begin{align*}
 |\{\gamma \in \X(\mC):\supp(\gamma) = U^\perp\}| 
&= \sum_{V\in [U,E]} \mu(U,V) |C^{(1)}_V|^2|C^{(2)}_V|^2 \\   &=\sum_{V\in [U,E]} \mu(U,V) q^{2(k_1-\rho_1(V))}q^{2(k_2-\rho_2(V))} = q^2 -3 +q = 3. \end{align*}
\end{example}

\bigskip

\section*{Acknowledgments}
The authors are thankful to the anonymous referees for their careful reading and constructive comments. In particular, their comments led to a correction in the proof of Proposition 4.9. 
G.~N.~A. was supported by the Swiss National Foundation through grant no. 210966.


\bigskip
\bibliographystyle{alphaurl}
\bibliography{references.bib}

\bigskip
\bigskip

\newpage
\appendix
\section{}\label{appendix}

In Figure \ref{fig:qpolylattice} we illustrate the minors described in Example \ref{ex:qpolylattice}, along with their characteristic polynomials.
\begin{figure}[ht!]
\centering
\begin{tabular}{|c|c|}
\hline
& \\
   \includegraphics{q-polylattice.pdf} & \includegraphics{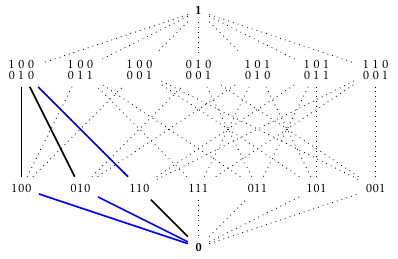}\\
   ${\mathbb P}(M;z)=z^5-2z^3-5z^2+6z$ &
   ${\mathbb P}(M|_H;z)=z^5-z^3-2z^2+2$ \\
   & \\
   \hline
   & \\
   \includegraphics{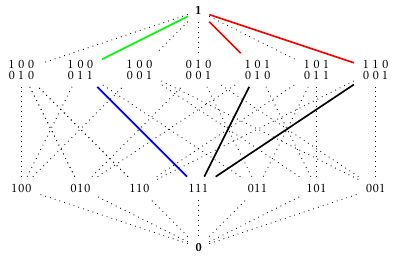}  &  \includegraphics{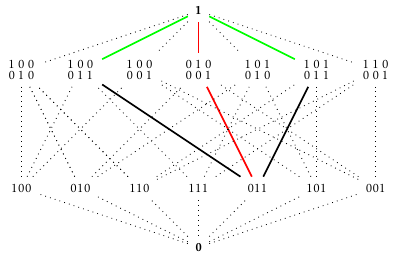}\\
   ${\mathbb P}(M/\langle 111 \rangle;z)=z^3-2z+1$&
   ${\mathbb P}(M/\langle 011 \rangle;z)=z^2-z$\\
   & \\
   \hline
   & \\
   \includegraphics{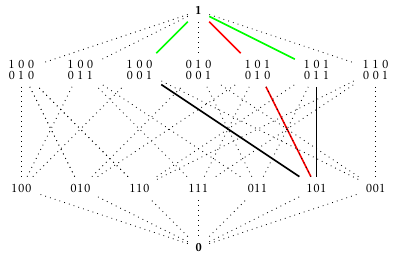}  &
   \includegraphics{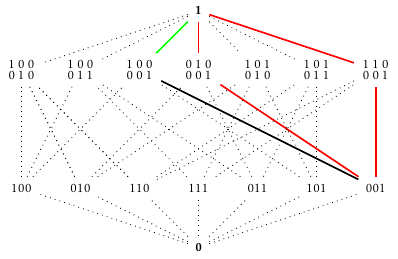}\\
   ${\mathbb P}(M/\langle 101 \rangle;z)=z^2-z$ &
   ${\mathbb P}(M/\langle 001 \rangle;z)=z^2-2z+1$\\
   & \\
   \hline
\end{tabular}
\caption{Minors of a $(2,3)$-polymatroid and their characteristic polynomials.}
 \label{fig:qpolylattice}
\end{figure}

\newpage

In Table \ref{tab:comparison_bounds} we provide some examples (computed using the \textsc{magma} software package \cite{magma}) to compare the lower bound that we provided in Proposition \ref{prop:LB_critical_exponent} with the actual critical exponent (CE) of some $q$-polymatroids induced by rank-metric codes, computed with the aid of the \textsc{magma} software package as evaluation of the characteristic polynomial. 

\bigskip 

\small{
\begin{table}[ht!]
\setlength\extrarowheight{5pt}
\begin{tabular}{|c|c|c|c|}
\hline
\textbf{$[n\times m,k,d]_q$} & \textbf{Code} & \textbf{$\left\lceil \frac{n}{m}\right\rceil$} & \textbf{CE} \\ \hline

 $\F_2$-$[6\times 3,5,2]$  & \begin{tabular}[c]{@{}c@{}}
 \tiny{$\Biggl\langle\begin{pmatrix}1 & 0 & 0 \\ 0 & 1 & 0 \\ 0 & 0 & 0 \\ 0 & 0 & 0 \\0 & 0 & 0 \\0 & 0 & 0  \end{pmatrix},
 \begin{pmatrix}0 & 0 & 0 \\ 0 & 1 & 0 \\ 0 & 0 & 1 \\ 0 & 0 & 0 \\ 0 & 0 & 0 \\ 0 & 0 & 0 \\  \end{pmatrix},
 \begin{pmatrix} 0 & 1 & 1 \\ 0 & 0 & 1 \\ 0 & 0 & 0 \\ 0 & 0 & 0 \\ 0 & 0 & 0 \\ 0 & 0 & 0 \\  \end{pmatrix},$
 $ \begin{pmatrix} 0 & 0 & 0 \\ 1 & 0 & 0 \\ 1 & 1 & 0 \\ 0 & 0 & 0 \\ 0 & 0 & 0 \\ 0 & 0 & 0 \end{pmatrix}, 
 \begin{pmatrix} 0 & 0 & 0 \\ 0 & 0 & 0 \\ 0 & 0 & 0 \\ 1 & 0 & 0 \\ 0 & 1 & 0 \\ 0 & 0 & 1  \end{pmatrix}\Biggr\rangle$} 
\end{tabular}  & $2$   & $3$   \\ \hline
$\F_2$-$[5\times 5,6,3]$   & \begin{tabular}[c]{@{}c@{}} \tiny{$\Biggl\langle\begin{pmatrix} 0 &1& 1 & 0 & 1 \\0 & 1 & 0 & 1 & 1 \\ 0 & 1 & 0 & 0 & 0 \\ 0& 0 & 1 & 1 & 0  \\ 0 & 0 & 0 & 0 & 0  \end{pmatrix},
\begin{pmatrix}0 & 0 & 0 & 0 & 0\\0 & 0 & 0 & 0 & 0\\ 0 & 0 & 0 & 0 & 1 \\ 0 & 0 & 0 & 1 & 0\\ 0 & 0 & 1 & 0 & 0\\  \end{pmatrix},
\begin{pmatrix}0 & 0 & 0 & 0 & 0\\ 0 & 0 & 0 & 0 & 1\\ 0 & 1 & 0 & 0 & 0 \\ 0 & 1 & 0 & 0 & 1\\ 0 & 0 & 1 & 1 & 1\\  \end{pmatrix},$\vspace{0.2cm}}\\
\tiny{$\begin{pmatrix} 0 & 0 & 0 & 1 & 1\\0 & 0 & 1& 0 & 0\\ 0 & 1 & 0 & 1 & 0 \\ 0 & 1 & 0 & 0 & 1\\ 0 & 0 & 1 & 0 & 0\\  \end{pmatrix},
\begin{pmatrix} 0 & 0 & 0 & 0 & 0\\ 0 & 1 & 1 & 1 & 0\\ 0 & 0 & 0 & 0 & 0 \\ 0 & 1 & 0 & 0 & 1\\ 0 & 1 & 0 & 1 & 0\\  \end{pmatrix},
\begin{pmatrix} 0 & 0 & 1 & 0 & 1\\ 0 & 0 & 1 & 0 & 0\\ 0 & 0 & 0 & 1 & 1 \\ 0 & 0 & 1 & 1 & 0\\ 0 & 0 & 1 & 1 & 0\\  \end{pmatrix}\Biggr\rangle$}
\end{tabular}  & $1$    & $2$   \\ \hline 

$\F_2$-$[4\times 4,5,3]$ &  \begin{tabular}[c]{@{}c@{}}\tiny{$\Biggl\langle\begin{pmatrix} 0 & 0 & 0  & 1 \\ 0 & 1 & 1 & 1 \\ 1 & 1 & 1 & 1 \\ 0 & 1 & 1 & 0 \end{pmatrix},
\begin{pmatrix}
1 & 1 & 0 & 0 \\ 1 & 0 & 0 & 1 \\ 0 & 0 & 1 & 0 \\ 0 & 0 & 1 & 0
\end{pmatrix},
\begin{pmatrix}
0 & 1 & 1 & 0 \\ 1 & 1 & 1 & 1 \\ 0 & 1 & 1 & 0 \\ 1 & 0 & 0 & 0
\end{pmatrix},$\vspace{0.2cm}}\\
\tiny{$\begin{pmatrix}
1 & 0 & 0 & 0 \\ 1 & 0 & 0 & 1 \\ 1 & 1 & 1 & 0 \\ 0 & 1 & 1 & 1 
\end{pmatrix},
\begin{pmatrix}
0 & 1 & 0 & 1 \\
1 & 0  &0 & 1 \\ 0 & 1 & 1 & 1 \\ 1 & 1 & 0 & 0
\end{pmatrix}\Biggr\rangle$}
\end{tabular}& $1$ & $2$\\\hline

$\F_2$-$[5\times 4,15,1]$ &\begin{tabular}[c]{@{}c@{}}\tiny{$\Biggl\langle\begin{pmatrix} 0 & 0 & 0 & 0  \\ 0 & 0 & 0  & 1 \\ 0 & 1 & 1 & 1 \\ 1 & 1 & 1 & 1 \\ 0 & 1 & 1 & 0 \end{pmatrix},
\begin{pmatrix}
0 & 0 & 0 & 0  \\ 1 & 1 & 0 & 0 \\ 1 & 0 & 0 & 1 \\ 0 & 0 & 1 & 0 \\ 0 & 0 & 1 & 0
\end{pmatrix},
\begin{pmatrix}
0 & 0 & 0 & 0  \\ 0 & 1 & 1 & 0 \\ 1 & 1 & 1 & 1 \\ 0 & 1 & 1 & 0 \\ 1 & 0 & 0 & 0
\end{pmatrix},$\vspace{0.2cm}}\\
\tiny{$
\begin{pmatrix}
0 & 0 & 0 & 0  \\ 1 & 0 & 0 & 0 \\ 1 & 0 & 0 & 1 \\ 1 & 1 & 1 & 0 \\ 0 & 1 & 1 & 1 
\end{pmatrix},
\begin{pmatrix}
0 & 0 & 0 & 0  \\ 0 & 1 & 0 & 1 \\
1 & 0  &0 & 1 \\ 0 & 1 & 1 & 1 \\ 1 & 1 & 0 & 0
\end{pmatrix}\Biggr\rangle^\perp$}
\end{tabular} & $2$ & $2$
\\
\hline
 $\F_2$-$[6\times 3,4,2]$  & \begin{tabular}[c]{@{}c@{}}
 \tiny{$\Biggl\langle\begin{pmatrix}1 & 0 & 0 \\ 0 & 0 & 0 \\ 0 & 1 & 1 \\ 0 & 0 & 0 \\0 & 0 & 0 \\0 & 0 & 0  \end{pmatrix},
 \begin{pmatrix}0 & 0 & 0 \\ 1 & 0 & 0 \\ 1 & 0 & 1 \\ 0 & 0 & 0 \\ 0 & 0 & 0 \\ 0 & 0 & 0 \\  \end{pmatrix},
 \begin{pmatrix} 0 & 0 & 0 \\ 0 & 0 & 1 \\ 1 & 1 & 1 \\ 0 & 0 & 0 \\ 0 & 0 & 0 \\ 0 & 0 & 0 \\  \end{pmatrix},$
 $ 
 \begin{pmatrix} 0 & 0 & 0 \\ 0 & 0 & 0 \\ 0 & 0 & 0 \\ 1 & 0 & 0 \\ 0 & 1 & 0 \\ 0 & 0 & 1  \end{pmatrix}\Biggr\rangle$}
\end{tabular}  & $2$   & $3$   \\ 
\hline\hline
\end{tabular}
\vspace*{5mm}
\caption{Critical exponents of some representable $q$-polymatroids}
 \label{tab:comparison_bounds}
\end{table}
}

\end{document}